\documentclass{amsart}
\usepackage{amsmath,amssymb}
\title{Exact sequences of Frobenius tensor categories}

\author[T. Shibata]{Taiki Shibata}
\address[T. Shibata]{Department of Applied Mathematics,
  Okayama University of Science \\
  1-1 Ridai-cho, Kita-ku Okayama-shi, Okayama 700-0005, Japan.}
\email{shibata@ous.ac.jp}
\author[K. Shimizu]{Kenichi Shimizu}
\address[K. Shimizu]{Department of Mathematical Sciences,
  Shibaura Institute of Technology \\
  307 Fukasaku, Minuma-ku, Saitama-shi, Saitama 337-8570, Japan.}
\email{kshimizu@shibaura-it.ac.jp}


\date{}

\usepackage{mathtools}
\usepackage[setpagesize=false]{hyperref}
\usepackage{tikz}
\usepackage{tikz-cd}
\usetikzlibrary{calc}
\usepackage{accents}
\usepackage{amsthm}
\numberwithin{equation}{section}
\theoremstyle{plain}
\newtheorem{C}{}[section] 
\newtheorem{lemma}[C]{Lemma}

\newtheorem{theorem}[C]{Theorem}

\theoremstyle{definition}
\newtheorem{definition}[C]{Definition}
\theoremstyle{remark}
\newtheorem{remark}[C]{Remark}
\newtheorem{example}[C]{Example}
\newcommand{\id}{\mathrm{id}}
\newcommand{\op}{\mathrm{op}}
\newcommand{\cop}{\mathrm{cop}}
\newcommand{\radj}{\mathrm{ra}}

\newcommand{\ladj}{\mathrm{la}}
\newcommand{\lladj}{\mathrm{lla}}

\newcommand{\bfk}{\Bbbk}

\newcommand{\Hom}{\mathrm{Hom}}

\newcommand{\Ker}{\mathrm{Ker}}
\newcommand{\Coker}{\mathrm{Coker}}
\newcommand{\Img}{\mathrm{Im}}
\newcommand{\Rep}{\mathrm{Rep}}

\newcommand{\coHom}{\mathrm{coHom}}

\newcommand{\rat}{\mathrm{rat}}

\newcommand{\unitobj}{\mathbf{1}}

\newcommand{\eval}{\mathrm{ev}}
\newcommand{\coev}{\mathrm{coev}}

\newcommand{\copow}{\otimes}

\newcommand{\Nak}{\mathbb{N}}

\newcommand{\Mod}{\mathfrak{M}}
\newcommand{\Vect}{\mathbf{Vec}}
\newcommand{\Sets}{\mathbf{Set}}
\newcommand{\Ind}{\mathrm{Ind}}
\newcommand{\fd}{\mathtt{fd}} 

\newcommand{\triv}{\mathrm{triv}}

\begin{document}

\begin{abstract}
  Given a tensor functor between tensor categories $\mathcal{C}$ and $\mathcal{D}$, we give criteria that, under certain assumptions, the Frobeniusness of $\mathcal{C}$ or $\mathcal{D}$ implies the Frobeniusness of the other one. 
  We also give an affirmative answer to Natale's question asking if the class of Frobenius tensor categories is closed under exact sequences.
\end{abstract}

\maketitle


\section{Introduction}

The algebra of functions on an affine group scheme is one of important examples of Hopf algebras from the beginning of the Hopf algebra theory.
When a Hopf algebra is viewed as an algebra of functions on a group, a cointegral on a Hopf algebra is explained as an abstraction of the linear functional on the space of functions defined by integrating a function with respect to a Haar measure.
A Hopf algebra admitting a non-zero cointegral is {\em co-Frobenius} as a coalgebra.
Co-Frobenius Hopf algebras has been studied extensively; see \cite{MR1786197} and references therein.

In the category of (co)modules over a Hopf algebra, one can form the tensor product and the dual of objects.
A {\em tensor category} \cite{MR3242743} generalizes such a situation.
Tensor categories appear many areas of mathematics and mathematical physics.
One of interesting directions of the study of tensor categories is to reconsider results on Hopf algebras from the viewpoint of tensor categories.
This approach also gives new insights to the theory of Hopf algebras.

A co-Frobenius Hopf algebra is characterized as a Hopf algebra $H$ such that every simple $H$-comodule has a finite-dimensional injective hull \cite{MR1786197}.
Accordingly, Andruskiewitsch, Cuadra and Etingof \cite{MR3410615} introduced the notion of a {\em Frobenius tensor category} as a tensor category whose every simple object admits an injective hull.
In \cite{2021arXiv211008739S}, we have extended some important results on co-Frobenius Hopf algebras to Frobenius tensor categories.
This paper aims to give further developments in this direction.
Specifically speaking, given a map $\phi: H \to K$ of Hopf algebras, Andruskiewitsch and Cuadra showed that, under certain assumptions, the co-Frobeniusness of $H$ or $K$ implies the co-Frobeniusness of the other one \cite[Corollary 2.9]{MR3032811}.
In this paper, we formulate a generalization of their result in terms of tensor categories and tensor functors, and prove the following theorem:

\begin{theorem}[= Theorem \ref{thm:Frobenius-and-tensor-functors}]
  \label{thm:intro-co-Fb-1}
  Let $F: \mathcal{C} \to \mathcal{D}$ be a tensor functor between tensor categories $\mathcal{C}$ and $\mathcal{D}$, and let $G$ be an ind-adjoint of $F$, that is, a right adjoint of the functor $\Ind(\mathcal{C}) \to \Ind(\mathcal{D})$ induced by $F$, where $\Ind(-)$ means the ind-completion \cite{MR2182076}.
  Then the following hold:
  \begin{enumerate}
  \item [(a)] If $\mathcal{C}$ is Frobenius and $G$ is exact, then $\mathcal{D}$ is Frobenius.
  \item [(b)] If $\mathcal{D}$ is Frobenius and $G(\mathcal{D}) \subset \mathcal{C}$, then $\mathcal{C}$ is Frobenius.
  \end{enumerate}
\end{theorem}

The notion of an exact sequence of tensor categories (see Subsection~\ref{subsec:exact-seq-ten-cat}) was introduced and studied by Brugui\`eres and Natale in \cite{MR2863377,MR3161401} as a generalization of an exact sequence of Hopf algebras.
Andruskiewitsch and Cuadra also showed that the class of co-Frobenius Hopf algebras is closed under exact sequences \cite[Theorem 2.10]{MR3032811}.
Recently, Natale asked if the same holds for exact sequences of tensor categories \cite[Question 6.12]{MR4281372}.
Our main result in this paper answers this question affirmatively and is stated as follows:

\begin{theorem}[= Theorem~\ref{thm:Frobenius-closed-under-exact-seq}]
  \label{thm:intro-co-Fb-2}
  Let $\mathcal{C}' \xrightarrow{\quad \iota \quad} \mathcal{C} \xrightarrow{\quad F \quad} \mathcal{D}$ be an exact sequence of tensor categories such that an ind-adjoint of $F$ is exact.
  Then $\mathcal{C}$ is Frobenius if and only if both $\mathcal{C}'$ and $\mathcal{D}$ are.
\end{theorem}

Andruskiewitsch and Cuadra proved the original Hopf algebraic version of Theorems~\ref{thm:intro-co-Fb-1} and \ref{thm:intro-co-Fb-2} based on the integral theory for Hopf algebras.
Their methods cannot be directly generalized to tensor categories because of the absence of the notion of cointegrals for tensor categories.
In this paper, we prove Theorems~\ref{thm:intro-co-Fb-1} and \ref{thm:intro-co-Fb-2} based on the theory of the Nakayama functor for locally finite abelian categories established in \cite{2021arXiv211008739S}.
We have demonstrated many results on Hopf algebras can be generalized to tensor categories by using the Nakayama functor in \cite{2021arXiv211008739S}.
In view of these successes, we might consider the Nakayama functor as a category-theoretical alternate for the notion of cointegrals.

Applications of Theorems~\ref{thm:intro-co-Fb-1} and \ref{thm:intro-co-Fb-2} are also given in this paper.
Let $\mathcal{C}$ be a semisimple braided tensor category, and let $H$ be a Hopf algebra in $\Ind(\mathcal{C})$ with invertible antipode.
We note that $H$ has a maximal quotient Hopf algebra $\overline{H} \in \Vect$, where $\Vect$ is the category of all vector spaces and regarded as a full subcategory of $\Ind(\mathcal{C})$.
Our first application concerns the Frobenius property of the tensor category $\mathcal{C}^H$ of right $H$-comodules whose underlying object belongs to $\mathcal{C}$.
As a generalization of the existence criterion for non-zero cointegrals on the coordinate superalgebra of an affine algebraic supergroup scheme \cite{MR4375528}, we give:

\begin{theorem}[= Theorem~\ref{thm:cointegral}]
  \label{thm:intro-3}
  Suppose that there is an object $W \in \mathcal{C}$ and an isomorphism $H \cong \overline{H} \otimes W$ of left $\overline{H}$-comodules.
  Then the tensor category $\mathcal{C}^H$ is Frobenius if and only if the Hopf algebra $\overline{H}$ is co-Frobenius.
\end{theorem}

The second application concerns de-equivariantization.
Let $G$ be an affine group scheme, and let $\mathcal{C}$ be a tensor category equipped with a `central' inclusion functor $\Rep(G) \to \mathcal{C}$.
Then the coordinate algebra $\mathcal{O}(G)$ is a commutative algebra in the Drinfeld center of $\Ind(\mathcal{C})$ in a natural way and hence the category $\Ind(\mathcal{C})_{\mathcal{O}(G)}$ of $\mathcal{O}(G)$-modules in $\Ind(\mathcal{C})$ is a monoidal category with respect to the tensor product over $\mathcal{O}(G)$.
Let $\mathcal{D}$ be the full subcategory of finitely generated objects of $\Ind(\mathcal{C})_{\mathcal{O}(G)}$.
The category $\mathcal{D}$ is a monoidal full subcategory of $\Ind(\mathcal{C})_{\mathcal{O}(G)}$ (see Subsection~\ref{subsec:de-equiv} for more details).
We prove:

\begin{theorem}[= Theorem~\ref{thm:de-equivari}]
  \label{thm:intro-4}
  Suppose that the monoidal category $\mathcal{D}$ is a tensor category.
  Then the following hold:
  \begin{enumerate}
  \item $\mathcal{D}$ coincides with the full subcategory of finitely presented objects of $\Ind(\mathcal{C})_{\mathcal{O}(G)}$.
  \item $\Ind(\mathcal{D})$ is equivalent to $\Ind(\mathcal{C})_{\mathcal{O}(G)}$.
  \item There is an exact sequence $\Rep(G) \to \mathcal{C} \to \mathcal{D}$ of tensor categories.
  \item $\mathcal{C}$ is a Frobenius tensor category if and only if $\mathcal{O}(G)$ is co-Frobenius and $\mathcal{D}$ is Frobenius.
  \end{enumerate}
\end{theorem}

We may call $\mathcal{D}$ of this theorem the {\em de-equivariantization} of $\mathcal{C}$ by $G$ ({\it cf}. \cite{MR2609644,MR4227163}).
Admittedly, in general, it is not easy to check whether $\mathcal{D}$ satisfies the assumption of this theorem.
Nevertheless, we think this theorem is noteworthy in view of future applications.
Indeed, the de-equivariantization construction (or `modularization') is used in \cite{2018arXiv180902116G,MR4227163} to construct new examples of non-semisimple modular tensor categories.

\subsection*{Organization of this paper}
This paper is organized as follows:
In Section \ref{sec:preliminaries}, we fix notations and recall from \cite{2021arXiv211008739S} basic results on the Nakayama functor for coalgebras and locally finite abelian categories.
Moreover, we provide technical lemmas characterizing quasi-co-Frobenius coalgebras in terms of the Nakayama functor (Lemmas \ref{lem:Nakayama-QcF} and \ref{lem:Nakayama-QcF-2}).

In Section \ref{sec:Fro-pro-ten-fun}, we prove Theorem \ref{thm:intro-co-Fb-1}.
We first recall the definition of Frobenius tensor categories from \cite{MR3410615} and quote some equivalent conditions for a tensor category to be Frobenius from \cite{2021arXiv211008739S} as Lemma \ref{lem:Frobenius-criterions-0}.
A useful criterion is that a tensor category is Frobenius if and only if it has a non-zero projective object, if and only if it has a non-zero injective object.
After discussing how the setting of \cite[Corollary 2.9]{MR3032811} is interpreted in terms of categories and functors, we prove Theorem \ref{thm:intro-co-Fb-1} by finding non-zero projective or injective objects by using an ind-adjoint of a tensor functor.

In Section \ref{sec:exact-seq-ten-cat}, we prove Theorem \ref{thm:intro-co-Fb-2}.
After recalling the definition of an exact sequence of tensor categories from  \cite{MR2863377,MR3161401}, we give a formula of the Nakayama functor for a tensor category and find that the Nakayama functor for a tensor category is exact (Lemma \ref{lem:Nakayama-exactness}).
By combining this result with Lemma \ref{lem:Nakayama-QcF-2}, we prove a key lemma stating that a tensor category $\mathcal{C}$ is Frobenius if and only if there are an injective object $E \in \Ind(\mathcal{C})$ and an object $V \in \mathcal{C}$ such that $\Hom_{\Ind(\mathcal{C})}(E, V) \ne 0$ (Lemma~\ref{lem:Frobenius-criterions}).
The `if' part of Theorem \ref{thm:intro-co-Fb-2} is proved by finding such objects $E$ and $V$ by using an ind-adjoint of a tensor functor.
The `only if' part follows from Theorem \ref{thm:intro-co-Fb-1} and a generalization of Sullivan's argument given in \cite{MR304418}.

In Section~\ref{sec:applications}, we give applications of our results.
Theorem \ref{thm:intro-3} is obtained by applying Theorem \ref{thm:intro-co-Fb-1} to the tensor functor $\mathcal{C}^H \to \mathcal{C}^{\overline{H}}$ induced by the quotient morphism $H \to \overline{H}$.
Theorem \ref{thm:intro-4} is proved in Subsection \ref{subsec:de-equiv}.
A large part of this subsection is rather devoted to explain that the category $\mathcal{D}$ in the statement of the theorem is a monoidal category.
Parts (1) and (2) of Theorem \ref{thm:intro-co-Fb-1} follow from basic results on ind-completions \cite{MR2182076} and the strong assumption of the theorem that $\mathcal{D}$ is a tensor category.
Part (3) is straightforward, and Part (4) follows immediately from Theorem \ref{thm:intro-4}.

Let $\mathcal{V}$ be a tensor category, and let $\mathbf{C}$ be a coalgebra in $\Ind(\mathcal{V})$.
In Appendix~\ref{sec:fundamental-theorem}, we show that the category $\Ind(\mathcal{V})^{\mathbf{C}}$ of all $\mathbf{C}$-comodules in $\Ind(\mathcal{V})$ is identified with the ind-completion of $\mathcal{V}^{\mathbf{C}}$ (Theorem \ref{thm:appendix-fund-thm-comodules}).
This theorem is used in the proof of Theorem \ref{thm:intro-3}.
Lyubashenko \cite[Section 2.3]{MR1625495} proved Theorem \ref{thm:appendix-fund-thm-comodules} in a general setting of squared coalgebras.
We include a direct proof of Theorem \ref{thm:appendix-fund-thm-comodules} as an appendix for the reader's convenience.

\subsection*{Acknowledgement}

The first author (T.S.) is supported by JSPS KAKENHI Grant Number JP22K13905.
The second author (K.S.) is supported by JSPS KAKENHI Grant Number JP20K03520.

\section{Preliminaries}
\label{sec:preliminaries}

\subsection{Basic notation}
\label{subsec:notation}

For the basics on category theory, we refer the book of Mac Lane \cite{MR1712872}. Given a category $\mathcal{C}$, we denote the opposite category of $\mathcal{C}$ by $\mathcal{C}^{\op}$.

Throughout this paper, we work over a field $\bfk$.
The category of all vector spaces (over the field $\bfk$) is denoted by $\Vect$.
The dual space of $X \in \Vect$ is denoted by $X^*$.
For $f \in X^*$ and $x \in X$, we often write $f(x) \in \bfk$ as $\langle f, x \rangle$.
By a (co)algebra, we always mean a (co)associative (co)unital (co)algebra over $\bfk$.
The comultiplication and the counit of a coalgebra $C$ are denoted by $\Delta$ and $\varepsilon$, respectively.
The Sweedler notation, such as $\Delta(c) = c_{(1)} \otimes c_{(2)}$, will be used to express the comultiplication of an element $c \in C$.
We note that $C^*$ is an algebra with respect to the convolution product defined by $\langle f g, c \rangle = \langle f, c_{(1)} \rangle \langle g, c_{(2)} \rangle$ for $f, g \in C^*$ and $c \in C$.

Given an algebra $A$, we denote by ${}_A \Mod$ and $\Mod_A$ the category of left and right $A$-modules, respectively.
Similarly, given a coalgebra $C$, we denote by ${}^C \Mod$ and $\Mod^C$ the category of left and right $C$-comodules, respectively.
The Hom functors of these four categories are written as ${}_A\Hom$, $\Hom_A$, ${}^C\Hom$ and $\Hom^C$, respectively.
We use the subscript `$\fd$' to mean the full subcategory consisting of all finite-dimensional objects.
For example, $\Mod^C_{\fd}$ means the category of finite-dimensional right $C$-comodules.

\subsection{Coalgebras and comodules}
\label{subsec:coalgebras}

For the basics on coalgebras and comodules, we refer the reader to \cite{MR1786197}.
Let $C$ be a coalgebra.
Given a right $C$-comodule $M$, we usually denote the coaction of $C$ on $M$ by $\delta_M: M \to M \otimes_{\bfk} C$ and express it as $\delta_M(m) = m_{(0)} \otimes m_{(1)}$ ($m \in M$).
A right $C$-comodule $M$ becomes a left $C^*$-module by the action given by $f \rightharpoonup m := m_{(0)} \langle f, m_{(1)} \rangle$ ($f \in C^*$, $m \in M$). A left $C^*$-module is said to be {\em rational} if it is obtained from a right $C$-comodule in this way. The {\em rational part} of $M \in {}_{C^*}\Mod$, denoted by $M^{\rat}$, is the maximal rational left $C^*$-module of $M$.
We may, and do, regard $M^{\rat}$ as a right $C$-comodule.
The rational part of a right $C^*$-module is defined in an analogous way.

We recall the following important classes of coalgebras:
\begin{enumerate}
\item A {\em left semiperfect coalgebra} is a coalgebra $C$ such that every simple left $C$-comodule has a projective cover in ${}^C\Mod$.
  A coalgebra $C$ is said to be {\em right semiperfect} if $C^{\cop}$ is left semiperfect (where $C^{\cop}$ is the opposite coalgebra).
  A {\em semiperfect coalgebra} is a left and right semiperfect coalgebra.
\item A {\em left quasi-co-Frobenius (QcF) coalgebra} is a coalgebra $C$ such that there is a cardinal $\kappa$ and an injective homomorphism $C \to (C^{*})^{\oplus \kappa}$ of left $C^*$-modules.
  A coalgebra $C$ is said to be {\em right QcF} if $C^{\cop}$ is left QcF.
  A {\em QcF coalgebra} is a left and right QcF coalgebra.
\item A {\em left co-Frobenius coalgebra} is a coalgebra $C$ such that there is an injective homomorphism $C \to C^{*}$ of left $C^*$-modules.
  A coalgebra $C$ is said to be {\em right co-Frobenius} if $C^{\cop}$ is left co-Frobenius.
  A {\em co-Frobenius coalgebra} is a left and right co-Frobenius.
\end{enumerate}

It is known that a left (right) co-Frobenius coalgebra is QcF, and a left (right) QcF coalgebra is left (right) semiperfect.
Most of results on these classes of coalgebras we will use in this paper are found in \cite{MR1786197}.
One exception is a characterization of QcF coalgebras given by Iovanov \cite[Lemma 1.7]{MR3150709}.

\subsection{Nakayama functor for coalgebras}

Let $C$ be a coalgebra.
Given a right $C$-comodule $M$, we regard $M$ as a left $C^*$-module in a way explained in Subsection~\ref{subsec:coalgebras} and define $\Nak_C(M) = C \otimes_{C^*} M$ as a vector space. There is a well-defined right $C$-coaction given by
\begin{equation*}
  \Nak_C(M) \to \Nak_C(M) \otimes_{\bfk} C,
  \quad c \otimes_{C^*} m \mapsto (c_{(1)} \otimes_{C^*} m) \otimes c_{(2)}
\end{equation*}
for $c \in C$ and $m \in M$.

\begin{definition}
  We define the {\em Nakayama functor} for $C$ by
  \begin{equation*}
    \Nak_C: \Mod^C \to \Mod^C, \quad M \mapsto C \otimes_{C^*} M.
  \end{equation*}
\end{definition}

Strictly speaking, this functor is the {\em right exact Nakayama functor} for $C$ introduced and investigated in \cite{2021arXiv211008739S}.
There is also a left exact variant, which is not needed in this paper.
For reader's convenience, we include some results on the Nakayama functor given in \cite{2021arXiv211008739S} as the following lemma:

\begin{lemma}
  \label{lem:Nakayama-basics}
  Let $C$ be a coalgebra. Then the following hold:
  \begin{enumerate}
  \item $\Nak_C$ has a right adjoint, and hence it preserves colimits.
  \item If $C$ is semiperfect, then the full subcategory $\Mod^C_{\fd}$ is closed under $\Nak_C$.
  \item $C$ is QcF if and only if $\Nak_C$ is an equivalence.
  \item $C$ is co-Frobenius if and only if $\Nak_C$ is an equivalence and preserves the dimension of simple right $C$-comodules.
  \end{enumerate}
\end{lemma}

We add some supplementary results for the purpose of this paper.
Let $C$ be a coalgebra. We define $C^{*\rat}_{\ell} \in \Mod^C$ and $C^{*\rat}_r \in {}^C\Mod$ to be the rational part of $C^*$ as a left and a right $C^*$-module, respectively.
We note:

\begin{lemma}
  \label{lem:Nakayama-dual}
  For $M \in \Mod^C$, there are natural isomorphisms
  \begin{equation*}
    \Nak_C(M)^*
    \cong {}_{C^*}\Hom(M, C^{*})
    \cong \Hom^C(M, C^{*\rat}_{\ell})
  \end{equation*}
  of vector spaces.
\end{lemma}
\begin{proof}
  The first isomorphism is just the tensor-Hom adjunction, and the second one follows from the fact that an image of a rational $C^*$-module is rational.
\end{proof}

A right $C$-comodule $Q$ is said to be {\em quasi-finite} \cite{MR472967} if $\Hom^C(X, Q)$ is finite-dimensional for all $X \in \Mod^C_{\fd}$.
By Lemma~\ref{lem:Nakayama-dual} and the fact that the contravariant endofunctor $X \mapsto X^*$ on $\Vect$ preserves and reflects exact sequences, one immediately proves:

\begin{lemma}
  \label{lem:Nakayama-and-rat-dual}
  Let $C$ be a coalgebra.
  \begin{itemize}
  \item [(a)] $\Nak_C$ is the zero functor if and only if $C^{*\rat}_{\ell} = 0$.
  \item [(b)] $\Nak_C$ is exact if and only if $C^{*\rat}_{\ell}$ is an injective right $C$-comodule.
  \item [(c)] $\Nak_C$ preserves $\Mod^C_{\fd}$ if and only if $C^{*\rat}_{\ell}$ is quasi-finite.
  \end{itemize}
\end{lemma}

We recall that a ring $R$ is said to be left self-injective if $R$ is injective as a left $R$-module.
By Lemma~\ref{lem:Nakayama-dual}, we also have:

\begin{lemma}
  \label{lem:Nakayama-dual-self-injective}
  $\Nak_C$ is exact if the ring $C^*$ is left self-injective.
\end{lemma}

The coverse of this lemma does not hold in general.
Indeed, let $C = \bfk[t]$ be the polynomial algebra over $\bfk$, and equip it with a Hopf algebra structure determined by $\Delta(t) = t \otimes 1 + 1 \otimes t$.
Then $C^*$ is isomorphic to the algebra $\bfk[[t]]$ of formal power series.
Since $C^*$ has no non-zero finite-dimensional submodules, we have $C^{*\rat}_{\ell} = 0$ by \cite[Corollary 2.2.16]{MR1786197}.
Thus $\Nak_C$ is the zero functor and, in particular, it is exact. However, $C^*$ is not left self-injective by the classical fact that a module over a principal ideal domain is injective if and only if it is a divisible module.

We also give the following new characterization of QcF coalgebras:

\begin{lemma}
  \label{lem:Nakayama-QcF}
  Let $C$ be a coalgebra such that
  \begin{equation}
    \label{eq:Nakayama-QcF-assumption}
    \text{the functor $\Nak_{C^{\cop}} : {}^C\Mod \to {}^C\Mod$ is exact},
  \end{equation}
  where ${}^C\Mod$ is identified with $\Mod^{C^{\cop}}$.
  Then the following assertions are equivalent:
  \begin{enumerate}
  \item $C$ is QcF.
  \item Every simple right $C$-comodule is an epimorphic image of an injective right $C$-comodule.
  \item Every simple right $C$-comodule is an epimorphic image of $C$.
  \end{enumerate}
\end{lemma}

We note that \eqref{eq:Nakayama-QcF-assumption} can be rephrased as follows:
For $D = C^{\cop}$, we have $D^{*\rat}_{\ell} = C^{*\rat}_{r}$ as subspaces of $D^*$ ($ = C^*$).
Thus, by Lemma~\ref{lem:Nakayama-and-rat-dual}, the assumption \eqref{eq:Nakayama-QcF-assumption} is equivalent to that $C^{*\rat}_{r}$ is injective as a left $C$-comodule.
The necessity of the assumption \eqref{eq:Nakayama-QcF-assumption} will be discussed in Remark \ref{rem:Nakayama-QcF-assumption}.

\begin{proof}[Proof of Lemma \ref{lem:Nakayama-QcF}]
  \underline{(1) $\Rightarrow$ (2)}.
  Suppose that (1) holds.
  Then every injective right $C$-comodule is projective, and vice varsa \cite[Theorem 3.3.4]{MR1786197}.
  Since a QcF coalgebra is semiperfect, every simple right $C$-comodule has a projective cover, which is in fact injective.
  Hence (2) holds.

  \medskip
  \underline{(2) $\Rightarrow$ (3)}.
  Suppose that (2) holds.
  Let $S$ be a simple right $C$-comodule, and let $\pi: E \to S$ be an epimorphism in $\Mod^C$ with $E$ injective.
  We decompose $E$ into the direct sum of indecomposable injective right $C$-comodules, as $E = \bigoplus_{i \in I} E_{i}$, and let $\pi_i : E_i \to S$ be the restriction of $\pi$ to $E_i$.
  We note that each $E_i$ is an injective hull of a simple right $C$-comodule and thus it is a direct summand of $C$ \cite[Theorem 2.4.16]{MR1786197}.
  Now we pick an element $i_0 \in I$ such that $\pi_{i_0} \ne 0$.
  Then $\pi_{i_0}$ is epic since $S$ is simple. Thus $S$ is an epimorphic image of $C$.

  \medskip
  \underline{(3) $\Rightarrow$ (1)}.
  Suppose that (3) holds. Let $\{ S_i \}_{i \in I}$ be a complete set of representatives of the isomorphism classes of simple left $C$-comodules, and let $E_i$ be an injective hull of $S_i$ in ${}^C\Mod$.
  We decompose injective left $C$-comodules $C$ and $C^{*\rat}_r$ into direct sums of $E_i$'s, as follows:
  \begin{equation*}
    C \cong \bigoplus_{i \in I} E_i^{\oplus a_i},
    \quad
    C^{*\rat}_r \cong \bigoplus_{i \in I} E_i^{\oplus b_i}.
  \end{equation*}

  We have $a_i > 0$ for all $i \in I$.
  For each $i \in I$, there is an epimorphism $C \to S_i^*$ in $\Mod^C$ by the assumption.
  Thus there is a monomorphism of right $C^*$-module from $S_i$ ($\cong S_i^{**}$) into $C^*$.
  Since the right $C^*$-module $S_i$ is rational, $C^{*\rat}_r$ has a subcomodule isomorphic to $S_i$.
  This implies $b_i > 0$.

  Let $\kappa$ be an infinite cardinal greater than all $a_i$'s and $b_i$'s.
  By the above discussion, we have $C^{\oplus \kappa} \cong \bigoplus_{i \in I} E_i^{\oplus \kappa} \cong (C^{*\rat}_r)^{\oplus \kappa}$ as left $C$-comodules. Now (1) follows from Iovanov's result \cite[Lemma 1.7]{MR3150709} stating that a coalgebra $Q$ is QcF if and only if there is a cardinal $\kappa > 0$ such that $(Q^{*\rat}_{r})^{\oplus \kappa} \cong Q^{\oplus \kappa}$ as left $Q$-comodules.
\end{proof}

\begin{remark}
  \label{rem:Nakayama-QcF-assumption}
  The assumption \eqref{eq:Nakayama-QcF-assumption} of Lemma \ref{lem:Nakayama-QcF} is not used to prove (1) $\Rightarrow$ (2) and (2) $\Rightarrow$ (3).
  In addition, it is obvious that (3) $\Rightarrow$ (2) holds without assuming  \eqref{eq:Nakayama-QcF-assumption}.
  To show (3) $\Rightarrow$ (1), the assumption \eqref{eq:Nakayama-QcF-assumption} is necessarily.
  Indeed, let $C$ be a right QcF coalgebra which is not left QcF (a concrete example of such a coalgebra is found in \cite[Example 3.3.7]{MR1786197}).
  Then $C$ is not QcF, however, the condition (3) holds by the fact that $C$ is a generator in $\Mod^C$ \cite[Corollary 3.3.10]{MR1786197}.

  For a right QcF coalgebra $C$ which is not left QcF, we also have that $\Nak_C$ is exact by Lemma \ref{lem:Nakayama-dual-self-injective} and the fact that the dual algebra of a right semiperfect coalgebra is left self-injective \cite[Corollary 3.3.9]{MR1786197}.
  This means that \eqref{eq:Nakayama-QcF-assumption} cannot be replaced with the assumption that $\Nak_C$ is exact.
\end{remark}

\subsection{Nakayama functor for locally finite abelian categories}

Let $\mathcal{A}$ be a linear category, and let $W \in \Vect$ and $M \in \mathcal{A}$ be objects. The {\em copower} of $M$ by $W$, denoted by $W \copow M$, is an object of $\mathcal{A}$ such that there is a natural isomorphism
\begin{equation*}
  \Hom_{\mathcal{A}}(W \copow M, N) \cong \Hom_{\bfk}(W, \Hom_{\mathcal{A}}(M, N))
\end{equation*}
for $N \in \mathcal{A}$. Suppose that the copower $W \copow M$ exists for all objects $W \in \Vect$ and $M \in \mathcal{A}$.
Then, for objects $M, N \in \mathcal{A}$, the {\em coHom} space $\coHom_{\mathcal{A}}(M, N)$ is defined as the vector space such that there is a natural isomorphism
\begin{equation*}
  \Hom_{\bfk}(\coHom_{\mathcal{A}}(M, N), W)
  \cong \Hom_{\mathcal{A}}(N, W \copow M)
\end{equation*}
for $W \in \Vect$.

These definitions are inspired by Takeuchi \cite{MR472967}, where the coHom functor for the category of comodules was introduced and used to establish Morita theory for coalgebras.
Let $C$ be a coalgebra.
The copower of $M \in \Mod^C$ by $W \in \Vect$ is just the tensor product $W \otimes_{\bfk} M$ of them with the coaction $\id_W \otimes \delta_M$.
We fix $M \in \Mod^C$.
Then, according to \cite{MR472967}, the coHom space from $M$ to $N$ exists for all $N \in \Mod^C$ if and only if $M$ is quasi-finite. Furthermore, if this is the case, the coHom space from $M$ to $N$ is given by
\begin{equation*}
  \coHom_{\Mod^C}(M, N) = \varinjlim_{\lambda \in \Lambda} \Hom^C(N_{\lambda}, M)^*
\end{equation*}
for $N \in \Mod^{C}$, where $\{ N_{\lambda} \}_{\lambda \in \Lambda}$ is the directed set of all finite-dimensional subcomodules of $N$.

A {\em locally finite abelian category} \cite{MR3242743} is an essentially small linear category $\mathcal{A}$ such that every object of $\mathcal{A}$ is of finite length and $\Hom_{\mathcal{A}}(X, Y)$ is of finite-dimensional for all objects $X, Y \in \mathcal{A}$. It is known that a locally finite abelian category is precisely a linear category that is equivalent to $\Mod^C_{\fd}$ for some coalgebra $C$ (see \cite{MR472967} and \cite[Theorem 1.9.15]{MR3242743}).

Given a category $\mathcal{A}$, we denote its ind-completion \cite{MR2182076} by $\Ind(\mathcal{A})$.
There is a canonical functor $\iota_{\mathcal{A}}: \mathcal{A} \to \Ind(\mathcal{A})$ by which we regard $\mathcal{A}$ as a full subcategory of $\Ind(\mathcal{A})$.
If $\mathcal{A} = \Mod^C_{\fd}$ for some coalgebra $C$, then we may identify $\Ind(\mathcal{A})$ and $\iota_{\mathcal{A}}$ with $\Mod^C$ and the inclusion functor, respectively.
Furthermore, an object of $\Ind(\mathcal{A})$ belongs to $\mathcal{A}$ precisely if it is of finite length.

In \cite[Theorem 3.6]{2021arXiv211008739S}, we have proved that the Nakayama functor is given by
\begin{equation*}
  \Nak_C(M) = \int^{X \in \Mod^C_{\fd}} \coHom_{\Mod^C}(X, M) \otimes_{\bfk} X
  \quad (M \in \Mod^C),
\end{equation*}
where the integral means a {\em coend} \cite{MR1712872}.
In view of this universal property of the Nakayama functor and the fact that a locally finite abelian category is equivalent to $\Mod^C_{\fd}$ for some coalgebra $C$, we define:

\begin{definition}[\cite{2021arXiv211008739S}]
  For a locally finite abelian category $\mathcal{A}$, we define the {\em Nakayama functor} to be the endofunctor $\Nak_{\Ind(\mathcal{A})}$ on $\Ind(\mathcal{A})$ given by
  \begin{equation*}
    \Nak_{\Ind(\mathcal{A})}(M) = \int^{X \in \mathcal{A}} \coHom_{\Ind(\mathcal{A})}(X, M) \copow X
    \quad (M \in \Ind(\mathcal{A})).
  \end{equation*}
\end{definition}

We note that the full subcategory $\mathcal{A} \subset \Ind(\mathcal{A})$ may not be closed under $\Nak_{\Ind(\mathcal{A})}$.
By Lemma \ref{lem:Nakayama-basics}, the functor $\Nak_{\Ind(\mathcal{A})}$ restricts to an endofunctor on $\mathcal{A}$ if there exists a semiperfect coalgebra $C$ such that $\mathcal{A} \approx \Mod^C_{\fd}$.

Given a functor $F$, we denote by $F^{\ladj}$ a left adjoint of $F$ (if it exists).
The following lemma was proved in \cite{2021arXiv211008739S} by using the universal property of the Nakayama functor in a similar way as the finite case established in \cite{MR4042867}.

\begin{lemma}[{\cite[Theorem 4.7]{2021arXiv211008739S}}]
  \label{lem:Nakayama-double-adj}
  Let $\mathcal{A}$ and $\mathcal{B}$ be locally finite abelian categories, and let $F: \Ind(\mathcal{A}) \to \Ind(\mathcal{B})$ be a linear functor such that $F$ preserves colimits and the double left adjoint $F^{\lladj} := (F^{\ladj})^{\ladj}$ exists. Then there is an isomorphism of functors $F \circ \Nak_{\Ind(\mathcal{A})} \cong \Nak_{\Ind(\mathcal{B})} \circ F^{\lladj}$.
\end{lemma}

For later use, we rephrase Lemma~\ref{lem:Nakayama-QcF} as follows:

\begin{lemma}
  \label{lem:Nakayama-QcF-2}
  Let $\mathcal{A}$ be a locally finite abelian category such that $\Nak_{\Ind(\mathcal{A}^{\op})}$ is exact, and let $Q$ be a coalgebra such that $\mathcal{A} \approx \Mod^{Q}_{\fd}$ as linear categories. Then the following are equivalent:
  \begin{enumerate}
  \item The coalgebra $Q$ is QcF.
  \item Every simple object of $\mathcal{A}$ is an epimorphic image of an injective object of $\Ind(\mathcal{A})$.
  \end{enumerate}
\end{lemma}
\begin{proof}
  Since $(\Mod^{Q}_{\fd})^{\op}$ and $\Mod^{Q^{\cop}}_{\fd}$ are equivalent via the functor $X \mapsto X^*$, the exactness of $\Nak_{\Ind(\mathcal{A}^{\op})}$ is equivalent to the exactness of $\Nak_{Q^{\cop}}$.
  Thus this lemma is obtained just by rephrasing Lemma~\ref{lem:Nakayama-QcF}.
\end{proof}

\section{Frobenius property and tensor functors}
\label{sec:Fro-pro-ten-fun}

\subsection{Tensor categories and tensor functors}

For basics on tensor categories, we refer the reader to \cite{MR3242743}.
We first fix our convention on monoidal categories:
All monoidal categories are assumed to be strict in view of Mac Lane's strictness theorem.
Given a monoidal category $\mathcal{C}$, we usually denote by $\otimes$ and $\unitobj$ the tensor product and the unit object of $\mathcal{C}$, respectively.
A left dual object of $X \in \mathcal{C}$ is an object $X^{\vee} \in \mathcal{C}$ equipped with morphisms $\eval_X : X^{\vee} \otimes X \to \unitobj$ and $\coev_X : \unitobj \to X \otimes X^{\vee}$ satisfying the zig-zag equations (see \cite[Section 2.10]{MR3242743} for the precise definition).
A right dual object of $X$, denoted by ${}^{\vee}\!X$, is an object having $X$ as a left dual object.
A monoidal category $\mathcal{C}$ is said to be {\em rigid} if every object of $\mathcal{C}$ has a left and a right dual object.

A simple object of a linear abelian category is said to be {\em absolutely simple} if its endomorphism algebra is isomorphic to the base field $\bfk$.
A {\em tensor category} is a locally finite abelian category $\mathcal{C}$ equipped with a structure of a rigid monoidal category such that the tensor product of $\mathcal{C}$ is bilinear and the unit object of $\mathcal{C}$ is absolutely simple.
Given two tensor categories $\mathcal{C}$ and $\mathcal{D}$, a {\em tensor functor} from $\mathcal{C}$ to $\mathcal{D}$ is a linear exact functor $F: \mathcal{C} \to \mathcal{D}$ equipped with a natural isomorphism
\begin{equation*}
  F(X) \otimes F(Y) \to F(X \otimes Y)
  \quad (X, Y \in \mathcal{C})
\end{equation*}
and an isomorphism $\unitobj \to F(\unitobj)$ satisfying certain equations.

\subsection{Frobenius tensor categories}

The main subject of this paper is the following class of tensor categories:

\begin{definition}
  A {\em Frobenius tensor category} \cite{MR3410615} is a tensor category $\mathcal{C}$ such that every object of $\mathcal{C}$ has an injective hull in $\mathcal{C}$.
\end{definition}

We recall that a coalgebra $Q$ is left semiperfect if and only if every finite-dimensional right $Q$-comodule has a finite-dimensional injective hull \cite[Theorem 3.2.3]{MR1786197}. Thus a Frobenius tensor category is the same thing as a tensor category that is linearly equivalent to $\Mod^{Q}_{\fd}$ for some left semiperfect coalgebra $Q$. It is known that a Hopf algebra is left semiperfect if and only if it is QcF \cite[Theorem 5.3.2]{MR1786197}. Motivated by this result and some other characterizations of co-Frobenius Hopf algebras, we have given the following characterizations of Frobenius tensor categories:

\begin{lemma}[{\cite[Theorem 5.7]{2021arXiv211008739S}}]
  \label{lem:Frobenius-criterions-0}
  For a tensor category $\mathcal{C}$, the following are equivalent:
  \begin{enumerate}
  \item $\mathcal{C}$ is Frobenius.
  \item $\mathcal{C}$ has a non-zero injective object.
  \item $\mathcal{C}$ has a non-zero projective object.
  \item $\mathcal{C} \approx \Mod^{Q}_{\fd}$ for some QcF coalgebra $Q$.
  \end{enumerate}
\end{lemma}

The Nakayama functor played an essential role in the proof. We will give further equivalence conditions for $\mathcal{C}$ to be Frobenius in Subsection \ref{subsec:Nakayama-fun} again using the Nakayama functor.

\subsection{Frobenius property and tensor functors}

Let $\phi: H \to K$ be a map of Hopf algebras, and regard $H$ as a right $K$-comodule by the coaction induced by $\phi$.
We say that a right $K$-comodule $M$ is {\em finitely cogenerated} if there are a positive integer $n$ and a monomorphism of right $K$-comodules from $M$ to $K^{\oplus n}$.
Andruskiewitsch and Cuadra \cite[Corollary 2.9]{MR3032811} gave the following criteria for $H$ or $K$ to be co-Frobenius:
\begin{enumerate}
\item [(a)] If $H$ is co-Frobenius and $H \in \Mod^K$ is injective, then $K$ is co-Frobenius.
\item [(b)] If $K$ is co-Frobenius and $H \in \Mod^K$ is finitely cogenerated, then $H$ is co-Frobenius.
\end{enumerate}

We aim to extend these results to the setting of tensor categories.
For this purpose, we interpret the conditions for $H$ in (a) and (b) in category-theoretical terms.
Let $\phi_{*} : {}^H\Mod \to {}^K\Mod$ be the functor induced by $\phi$ between the categories of left comodules.
This functor has a right adjoint
\begin{equation*}
  \phi_{*}^{\radj} := H \mathop{\square}\nolimits_{K} (-) : {}^{K}\Mod \to {}^{H}\Mod,
\end{equation*}
where $\mathop{\square}\nolimits_{K}$ means the cotensor product over $K$ \cite[\S2.3]{MR1786197}.
By \cite[Exercise 2.4.23]{MR1786197}, $H$ is injective as a right $K$-comodule if and only if $\phi_{*}^{\radj}$ is exact.

We assume that $H \in \Mod^K$ is finitely cogenerated. By definition, there are a positive integer $n$ and a monomorphism $H \hookrightarrow K^{\oplus n}$ of right $K$-comodules. Since the cotensor product preserves monomorphisms, we have a monomorphism
\begin{equation*}
  \phi_{*}^{\radj}(M) = H \mathop{\square}\nolimits_{K} M \hookrightarrow K^{\oplus n} \mathop{\square}\nolimits_{K} M \cong M^{\oplus n}
\end{equation*}
of vector spaces. Hence, $\phi_{*}^{\radj}(M)$ is finite-dimensional for all $M \in {}^K\Mod_{\fd}$.
Although it is slightly off-topic, we give a characterization when the functor $\phi_{*}^{\radj}$ preserves finite-dimensionality.

\begin{lemma}
  For a map $\phi: H \to K$ of Hopf algebras, the following are equivalent:
  \begin{enumerate}
  \item The vector space $\phi_{*}^{\radj}(M)$ is finite-dimensional for all $M \in {}^K\Mod_{\fd}$.
  \item $H$ is quasi-finite as a left $K$-comodule.
  \end{enumerate}
\end{lemma}

Although we do not know any concrete counterexamples,
we think that the equivalent conditions (1) and (2) of this lemma are strictly weaker than the condition that $H$ is finitely cogenerated as a right $K$-comodule.

\begin{proof}
  The assignment $X \mapsto X^*$ gives rise to an antiequivalence between ${}^K\Mod_{\fd}^{}$ and $\Mod^K_{\fd}$, and there is a natural isomorphism $H \mathop{\square}\nolimits_K X \cong {}^K\Hom(X^*, H)$ of vector spaces for $X \in {}^K\Mod_{\fd}^{}$ \cite[Proposition 2.3.7]{MR1786197}.
  The claim easily follows from this isomorphism and the definition of quasi-finiteness.
\end{proof}

Taking the above consideration on $\phi_{*}^{\radj}$ into account, we extend the result of Andruskiewitsch and Cuadra \cite[Corollary 2.9]{MR3032811} to the setting of tensor categories as follows:
Let $\mathcal{C}$ and $\mathcal{D}$ be tensor categories, and let $F: \mathcal{C} \to \mathcal{D}$ be a tensor functor.
Since $F$ is right exact by definition, it admits an ind-adjoint, that is, a right adjoint of the functor $\Ind(\mathcal{C}) \to \Ind(\mathcal{D})$ induced by $F$.

\begin{theorem}
  \label{thm:Frobenius-and-tensor-functors}
  Let $F : \mathcal{C} \to \mathcal{D}$ be a tensor functor, and let $G$ be an ind-adjoint of $F$. Then the following hold:
  \begin{enumerate}
  \item [(a)] If $\mathcal{C}$ is Frobenius and $G$ is exact, then $\mathcal{D}$ is Frobenius.
  \item [(b)] If $\mathcal{D}$ is Frobenius and $G(\mathcal{D}) \subset \mathcal{C}$, then $\mathcal{C}$ is Frobenius.
  \end{enumerate}
\end{theorem}
\begin{proof}
  (a) Suppose that $\mathcal{C}$ is Frobenius and $G$ is exact.
  Let $P \in \mathcal{C}$ be a non-zero projective object (which exists by Lemma \ref{lem:Frobenius-criterions-0}).
  We note that $P$ is projective also in $\Ind(\mathcal{C})$ (see \cite[Lemma 2.4]{2021arXiv211008739S}).
  There is an isomorphism
  \begin{equation*}
    \Hom_{\mathcal{D}}(F(P), -)
    \cong \Hom_{\Ind(\mathcal{C})}(P, G(-))
    = \Hom_{\Ind(\mathcal{C})}(P, -) \circ G
  \end{equation*}
  of functors from $\mathcal{D}$ to $\Vect$.
  Since $G$ is exact, and since $F$ is faithful, $F(P)$ is a non-zero projective object of $\mathcal{D}$.
  By Lemma~\ref{lem:Frobenius-criterions-0}, $\mathcal{D}$ is Frobenius.

  (b) Suppose that $\mathcal{D}$ is Frobenius and $G(\mathcal{D}) \subset \mathcal{C}$.
  By the latter assumption, we may, and do, regard $G$ as a functor from $\mathcal{D}$ to $\mathcal{C}$.
  If $G(E) = 0$ for all injective objects $E \in \mathcal{D}$, then $G = 0$ as a functor since $\mathcal{D}$ has enough injective objects and $G$ is left exact as a right adjoint of $F$. Hence we have $F = 0$, a contradiction. Thus there is an injective object $E \in \mathcal{D}$ such that $G(E) \ne 0$.
  We pick such an object $E$. Then there is an isomorphism
  $\Hom_{\mathcal{C}}(-, G(E)) \cong \Hom_{\mathcal{D}}(-, E) \circ F$
  of functors from $\mathcal{C}$ to $\Vect$. Since $F$ is exact, and since $E$ is injective, we conclude that $G(E)$ is a non-zero injective object of $\mathcal{C}$.
  By Lemma~\ref{lem:Frobenius-criterions-0}, $\mathcal{C}$ is Frobenius.
\end{proof}

\section{Exact sequence of Frobenius tensor categories}
\label{sec:exact-seq-ten-cat}

\subsection{Exact sequence of tensor categories}
\label{subsec:exact-seq-ten-cat}

An exact sequence of tensor categories \cite{MR2863377,MR3161401,MR4281372} is a sequence of tensor functors giving a category-theoretical interpretation of an exact sequence of Hopf algebras.
Natale \cite[Question 6.12]{MR4281372} asked if the class of Frobenius tensor categories is closed under exact sequences.
The goal of this section is to give an affirmative answer to this question (Theorem~\ref{thm:Frobenius-closed-under-exact-seq}).
The Nakayama functor plays a crucial role in our approach.

Let $\mathcal{C}$ be a tensor category.
An object of $\mathcal{C}$ is said to be {\em trivial} if it is isomorphic to a direct sum of finitely many copies of the unit object.
Given an object $X \in \mathcal{C}$, we denote by $X_{\triv}$ the maximal trivial subobject of $X$ (which exists because $X$ has finite length).
We recall from \cite{MR2863377} the following definition:

\begin{definition}
  An {\em exact sequence} of tensor categories is a sequence
  \begin{equation*}
    \mathcal{C}' \xrightarrow{\quad \iota \quad} \mathcal{C} \xrightarrow{\quad F \quad} \mathcal{D}      
  \end{equation*}
  of tensor functors between tensor categories $\mathcal{C}$, $\mathcal{C}'$ and $\mathcal{D}$ such that the following conditions are satisfied:
  \begin{enumerate}
  \item The functor $\iota$ is fully faithful, so that we may regard $\mathcal{C}'$ as a tensor full subcategory of $\mathcal{C}$.
  \item The functor $F$ is {\em dominant} in the sense that every object of $\mathcal{D}$ is a subobject of an object of the form $F(X)$ for some $X \in \mathcal{C}$.
  \item The functor $F$ is {\em normal} in the sense that every object $X \in \mathcal{C}$ has a subobject $X_0$ such that $F(X_0) = F(X)_{\triv}$ as subobjects of $F(X)$.
  \item The kernel of $F$, defined and denoted by
    \begin{equation*}
      \Ker(F) = \{ X \in \mathcal{C} \mid \text{$F(X)$ is trivial} \},
    \end{equation*}
    coincides with the essential image of the tensor functor $\iota$.
  \end{enumerate}
\end{definition}

A tensor functor $F: \mathcal{C} \to \mathcal{D}$ is dominant if and only if every object of $\mathcal{D}$ is a quotient of an object of the form $F(X)$ for some $X \in \mathcal{C}$, if and only if an ind-adjoint of $F$ is faithful \cite[Lemma 3.1]{MR2863377}.

The following characterization of the normality is useful:

\begin{lemma}
  \label{lem:normality-BN11-Prop-3-5}
  Let $F: \mathcal{C} \to \mathcal{D}$ be a tensor functor between tensor categories $\mathcal{C}$ and $\mathcal{D}$, and let $G$ be an ind-adjoint of $F$. Then $F$ is normal if and only if every subobject of $G(\unitobj)$ of finite length belongs to $\Ker(F)$.
\end{lemma}
\begin{proof}
  This lemma has been proved in \cite[Proposition 3.5]{MR2863377} under the assumption that $F$ has a right adjoint, or, equivalently, $G(\mathcal{D}) \subset \mathcal{C}$.
  Despite that $G$ is only an ind-adjoint in our setting, we can prove this lemma in a similar way as \cite{MR2863377}.
  For the sake of completeness, we include a proof.

  To prove this lemma, it is convenient to paraphrase the normality of $F$ as follows:
  Given an object $X$ of a tensor category, we denote by $\overline{X}_{\triv}$ the largest trivial quotient of $X$ (which exists because $X$ has finite length).
  By the duality, the object $\overline{X}_{\triv}$ may be identified with ${}^{\vee}((X^{\vee})_{\triv})$.
  Since a tensor functor preserves duals, the normality of $F$ is equivalent to the following proposition: For every object $X \in \mathcal{C}$, there is a quotient object $X_0$ of $X$ such that $F(X_0) = \overline{F(X)}_{\triv}$ as quotient objects of $F(X)$.

  \medskip \underline{The `only if' part}.
  We assume that $F$ is normal.
  Let $X$ be a subobject of $G(\unitobj)$ of finite length.
  We aim to show that $X$ belongs to $\Ker(F)$.
  By the normality of $F$, there is an epimorphism $q: X \to X_0$ in $\mathcal{C}$ such that $F(X_0) = \overline{F(X)}_{\triv}$ as quotient objects of $F(X)$.
  We consider the following commutative diagram:
  \begin{equation}
    \label{eq:normality-BN11-Prop-3-5-proof}
    \begin{tikzcd}
      \Hom_{\mathcal{D}}(F(X_0), \unitobj)
      \arrow[r, "\psi"]
      \arrow[d, "{\text{adjunction}}"', "\cong"]
      & \Hom_{\mathcal{D}}(F(X), \unitobj)
      \arrow[d, "{\text{adjunction}}", "\cong"'] \\
      \Hom_{\Ind(\mathcal{C})}(X_0, G(\unitobj))
      \arrow[r, "\phi"]
      & \Hom_{\Ind(\mathcal{C})}(X, G(\unitobj)),
    \end{tikzcd}
  \end{equation}
  where $\psi = \Hom_{\mathcal{D}}(F(q), \unitobj)$ and $\phi = \Hom_{\Ind(\mathcal{C})}(q, G(\unitobj))$.
  Since $F(X_0) = \overline{F(X)}_{\triv}$, the map $\psi$ in the above diagram is bijective, and thus so is $\phi$.
  Let $i: X \to G(\unitobj)$ be the inclusion morphism.
  By the above discussion, there exists a morphism $f: X_0 \to G(\unitobj)$ in $\mathcal{C}$ such that $i = f \circ q$. Since $i$ is monic, so is $q$. Thus $q$ is in fact an isomorphism. This means that $F(X)$ is trivial.

  \medskip \underline{The `if' part}.
  We assume that every subobject of $G(\unitobj)$ of finite length belongs to $\Ker(F)$.
  Let $X \in \mathcal{C}$ be an arbitrary object, and let $X_0 \in \mathcal{C}$ be the largest quotient object of $X$ belonging to $\Ker(F)$.
  With $q : X \to X_0$ the quotient morphism, we consider the commutative diagram \eqref{eq:normality-BN11-Prop-3-5-proof} above.
  By the definition of $X_0$, the map $\phi$ in the diagram is bijective, and thus so is $\psi$.
  In other words, every morphism $F(X) \to \unitobj$ factors uniquely through $F(q)$.
  This means that $F(X_0) = \overline{F(X)}_{\triv}$ as quotient objects of $F(X)$.
  The proof is done.
\end{proof}

\subsection{Nakayama functor of tensor categories}
\label{subsec:Nakayama-fun}

\newcommand{\DD}{\mathbb{D}}

Let $\mathcal{C}$ be a tensor category, and let $\Nak := \Nak_{\Ind(\mathcal{C})} : \Ind(\mathcal{C}) \to \Ind(\mathcal{C})$ be the Nakayama functor.
As shown in \cite{MR4042867,2021arXiv211008739S}, if $\mathcal{C}$ is Frobenius, then we have $\Nak(X) = \Nak(\unitobj) \otimes X^{\vee\vee}$ for $X \in \mathcal{C}$.
In this subsection, we first prove that the same formula for $\Nak(X)$ holds without the assumption that $\mathcal{C}$ is Frobenius.

We note that the category $\Ind(\mathcal{C})$ has a natural structure of a monoidal category such that the tensor product is linear, exact and cocontinuous.
Let $\DD$ denote the endofunctor on $\Ind(\mathcal{C})$ induced by the double dual functor $(-)^{\vee\vee}$ on $\mathcal{C}$. In other words, $\DD$ is the functor defined by $\DD(M) = \varinjlim_{\lambda \in \Lambda} M_{\lambda}^{\vee\vee}$ for an object $M \in \Ind(\mathcal{C})$ expressed as $M = \varinjlim_{\lambda \in \Lambda} M_{\lambda}$ for some filtered system $\{ M_{\lambda} \}_{\lambda \in \Lambda}$ of objects of $\mathcal{C}$.
Since the functor $(-)^{\vee\vee}$ is an autoequivalence of $\mathcal{C}$, the functor $\DD$ is an autoequivalence of $\Ind(\mathcal{C})$.
We now prove:

\begin{lemma}
  \label{lem:Nakayama-exactness}
  There are natural isomorphisms
  \begin{equation*}
    \Nak(V) \otimes \DD(W)
    \cong \Nak(V \otimes W) \cong
    \DD^{-1}(V) \otimes \Nak(W)
    \quad (V, W \in \Ind(\mathcal{C})).
  \end{equation*}
  In particular, we have $\Nak(X) \cong \Nak(\unitobj) \otimes \DD(X)$ for $X \in \Ind(\mathcal{C})$ and therefore $\Nak$ is exact.
\end{lemma}
\begin{proof}
  By applying Lemma~\ref{lem:Nakayama-double-adj} to $F = (-) \otimes X^{\vee\vee}$, we have the following natural isomorphism:
  \begin{equation*}
    \Nak(V \otimes X) \cong \Nak(V) \otimes X^{\vee\vee}
    \quad (V \in \Ind(\mathcal{C}), X \in \mathcal{C}).
  \end{equation*}
  Hence, if $W$ is an object of $\Ind(\mathcal{C})$ such that $W = \varinjlim_{\lambda \in \Lambda} W_{\lambda}$ for some $W_{\lambda} \in \mathcal{C}$, then we have
  \begin{equation*}
    \Nak(V \otimes W)
    \cong \varinjlim_{\lambda \in \Lambda} \Nak(V \otimes W_{\lambda})
    \cong \varinjlim_{\lambda \in \Lambda} \Nak(V) \otimes W_{\lambda}^{\vee\vee}
    \cong \Nak(V) \otimes \DD(W),
  \end{equation*}
  where we have used the fact that $\Nak$ and $\otimes$ preserve colimits.
  Thus we have established the first isomorphism of the statement of this lemma.
  The second one is given in a similar way.
\end{proof}

It is known that a Hopf algebra $H$ is co-Frobenius if and only if $H^{*\rat}_{\ell} \ne 0$ \cite[Corollary 5.2.4]{MR1786197}.
Lemma \ref{lem:Nakayama-and-rat-dual} says that $H^{*\rat}_{\ell} \ne 0$ is equivalent to that the Nakayama functor $\Nak_{H}$ is non-zero.
We extend this observation to tensor categories as follows:

\begin{lemma}
  \label{lem:Frobenius-criterions}
  For a tensor category $\mathcal{C}$, the following are equivalent:
  \begin{enumerate}
  \item $\mathcal{C}$ is Frobenius.
  \item $\Nak_{\Ind(\mathcal{C})}$ is an equivalence.
  \item $\Nak_{\Ind(\mathcal{C})}$ is a non-zero functor.
  \item There are an injective object $E \in \Ind(\mathcal{C})$ and an object $V \in \mathcal{C}$ such that
    \begin{equation*}
      \Hom_{\Ind(\mathcal{C})}(E, V) \ne 0.
    \end{equation*}
  \end{enumerate}
\end{lemma}
\begin{proof}
  We fix a coalgebra $Q$ such that $\mathcal{C} \approx \Mod^Q_{\fd}$ as linear categories.
  If (1) holds, then $Q$ is QcF by Lemma~\ref{lem:Frobenius-criterions-0}, and therefore (2) holds by Lemma~\ref{lem:Nakayama-basics}.
  The implication (2) $\Rightarrow$ (3) is trivial.
  Below we prove (3) $\Rightarrow$ (4) and (4) $\Rightarrow$ (1).

  \medskip
  \underline{(3) $\Rightarrow$ (4)}. Suppose that (3) holds.
  Then, by Lemma~\ref{lem:Nakayama-and-rat-dual}, we have $Q^{*\rat}_{\ell} \ne 0$.
  We pick a non-zero element $f$ of $Q^{*\rat}_{\ell}$.
  According to a characterization of the rational part \cite[Corollary 2.2.16]{MR1786197}, the kernel of $f$ contains a right $Q$-subcomodule $K \subset Q$ of finite codimension. Since $f \ne 0$, we have $Q/K \ne 0$. Now let $E$ and $V$ be objects of $\Ind(\mathcal{C})$ corresponding to $Q$ and $Q/K$, respectively, through the equivalence $\Ind(\mathcal{C}) \approx \Mod^Q$. Then $E$ is injective, $V$ is finite and there is a non-zero morphism from $E$ to $V$ corresponding to the quotient map $Q \to Q/K$.

  \medskip
  \underline{(4) $\Rightarrow$ (1)}.
  We note that $\mathcal{C}^{\op}$ has a natural structure of a tensor category. Thus, by Lemma~\ref{lem:Nakayama-exactness}, the functor $\Nak_{\Ind(\mathcal{C}^{\op})}$ is exact.
  In view of Lemma~\ref{lem:Nakayama-QcF-2}, it is enough to show that every simple object of $\mathcal{C}$ is an epimorphic image of an injective object of $\Ind(\mathcal{C})$.

  Now we suppose that (4) holds and pick an injective object $E \in \Ind(\mathcal{C})$ and an object of $V \in \mathcal{C}$ such that there is a non-zero morphism $\pi: E \to V$ in $\Ind(\mathcal{C})$. By replacing $V$ with the image of $\pi$, we may assume that $\pi$ is an epimorphism. We note that $X \otimes E$ is injective for all $X \in \mathcal{C}$, since
  \begin{equation*}
    \Hom_{\mathcal{C}}(-, X \otimes E)
    \cong \Hom_{\mathcal{C}}(X^{\vee} \otimes (-), E)
  \end{equation*}
  as functors from $\Ind(\mathcal{C})$ to $\Vect$. Thus every simple object $S \in \mathcal{C}$ is an epimorphic image of an injective object of $\Ind(\mathcal{C})$ as follows:
  \begin{equation*}
    S \otimes V^{\vee} \otimes E
    \xrightarrow{\quad \id_S \otimes \id_{V^{\vee}} \otimes \pi \quad} S \otimes V^{\vee} \otimes V
    \xrightarrow{\quad \id_S \otimes \eval_V \quad} S.
  \end{equation*}
  The proof is done.
\end{proof}

\subsection{Exact sequence of Frobenius tensor categories}
Let $H$, $K$ and $H'$ be Hopf algebras with bijective antipodes, and let
\begin{equation}
  \label{eq:exact-seq-Hopf}
  H' \xrightarrow{\quad i \quad} H \xrightarrow{\quad \phi \quad} K
\end{equation}
be a sequence of maps of Hopf algebras. It is known that the sequence \eqref{eq:exact-seq-Hopf} is exact if and only if the induced sequence
\begin{equation*}
  \Mod^{H'}_{\fd} \xrightarrow{\quad i_{*} \quad}
  \Mod^H_{\fd} \xrightarrow{\quad \phi_{*} \quad} \Mod^{K}_{\fd}
\end{equation*}
of tensor functors is an exact sequence of tensor categories. Andruskiewitsch and Cuadra \cite[Theorem 2.10]{MR3032811} showed that, provided that $H$ is faithfully coflat as a right $K$-comodule, $H$ is co-Frobenius if and only if both $H'$ and $K$ are.
Natale \cite[Question 6.12]{MR4281372} asked if the same holds for tensor categories.
The assumption that $H$ is faithfully coflat as a right $K$-comodule is, by definition, equivalent to that the functor $\phi_{*}^{\radj} = H \mathop{\square}\nolimits_K (-) : {}^K\Mod \to {}^H\Mod$ is exact and faithful.
In view of this observation, we give an affirmative answer to Natale's question in the following form:

\begin{theorem}
  \label{thm:Frobenius-closed-under-exact-seq}
  Let $\mathcal{C}' \xrightarrow{\quad \iota \quad} \mathcal{C} \xrightarrow{\quad F \quad} \mathcal{D}$ be an exact sequence of tensor categories, and let $G$ be an ind-adjoint of $F$. Suppose that $G$ is exact. Then the following are equivalent:
  \begin{enumerate}
  \item $\mathcal{C}$ is Frobenius.
  \item Both $\mathcal{C}'$ and $\mathcal{D}$ are Frobenius.
  \end{enumerate}
\end{theorem}
\begin{proof}
  For simplicity, we assume that $\mathcal{C}'$ is a full subcategory of $\mathcal{C}$ through the tensor functor $\iota$.
  Then $\Ind(\mathcal{C}')$ can be regarded as the full subcategory of $\Ind(\mathcal{C})$ consisting of all objects of $\Ind(\mathcal{C})$ written as a filtered colimit of objects of $\mathcal{C}'$.

  \medskip
  \underline{(1) $\Rightarrow$ (2)}.
  We suppose that $\mathcal{C}$ is Frobenius.
  By Theorem~\ref{thm:Frobenius-and-tensor-functors}, $\mathcal{D}$ is Frobenius.
  We show that $\mathcal{C}'$ is Frobenius by an argument inspired by Sullivan's \cite[Theorem 2.15]{MR304418}.
  Let $E_0$ be an injective hull of the unit object $\unitobj$ in $\Ind(\mathcal{C}')$, and let $E$ be an injective hull of $E_0$ in $\Ind(\mathcal{C})$.
  Then we have a sequence $\unitobj \hookrightarrow E_0 \hookrightarrow E$ of essential extensions in $\Ind(\mathcal{C})$.
  Thus $E$ is an injective hull of $\unitobj$ in $\Ind(\mathcal{C})$.
  Since an injective hull of a simple object of $\mathcal{C}$ is an object of $\mathcal{C}$ by the definition of a Frobenius tensor category, $E$ is in fact an object of $\mathcal{C}$.
  Thus $E_0$ belongs to $\mathcal{C}$ as a subobject of $E$.
  In conclusion, we have $E_0 \in \Ind(\mathcal{C}') \cap \mathcal{C} = \mathcal{C}'$.
  Thus, by Lemma~\ref{lem:Frobenius-criterions-0}, $\mathcal{C}'$ is Frobenius.

  \medskip
  \underline{(2) $\Rightarrow$ (1)}.
  We assume that both $\mathcal{C}'$ and $\mathcal{D}$ are Frobenius.
  By Lemma~\ref{lem:normality-BN11-Prop-3-5}, every subobject of $G(\unitobj)$ of finite length belongs to $\Ker(F) = \mathcal{C}'$.
  In other words, we have $G(\unitobj) \in \Ind(\mathcal{C}')$.
  The object $G(\unitobj)$ is actually an injective object of $\Ind(\mathcal{C}')$. Indeed, let $0 \to X \to Y \to Z \to 0$ be an exact sequence in $\mathcal{C}'$. Since $F$ is exact, we have an exact sequence $0 \to F(X) \to F(Y) \to F(Z) \to 0$ in $\mathcal{D}$.
  This sequence splits since each term is a trivial object of $\mathcal{D}$.
  Thus we obtain the split exact sequence
  \begin{equation*}
    0 \to \Hom_{\mathcal{D}}(F(Z), \unitobj) \to \Hom_{\mathcal{D}}(F(Y), \unitobj) \to \Hom_{\mathcal{D}}(F(X), \unitobj) \to 0,
  \end{equation*}
  which is isomorphic to the sequence
  \begin{equation*}
    0 \to \Hom_{\Ind(\mathcal{C})}(Z, G(\unitobj)) \to \Hom_{\Ind(\mathcal{C})}(Y, G(\unitobj)) \to \Hom_{\Ind(\mathcal{C})}(X, G(\unitobj)) \to 0.
  \end{equation*}
  Therefore the contravariant functor $\Hom_{\Ind(\mathcal{C})}(-, G(\unitobj)) : \Ind(\mathcal{C}) \to \Vect$ preserves exact sequences in $\mathcal{C}'$.
  We conclude that $G(\unitobj)$ is an injective object of $\Ind(\mathcal{C}')$ by the following fact: A right comodule $E$ over a coalgebra $Q$ is injective if and only if $\Hom^Q(-, E)$ preserves exact sequences of finite-dimensional right $Q$-comodules \cite[Theorem 2.4.17]{MR1786197}.

  We have $G(\unitobj) \ne 0$ since $\Hom_{\Ind(\mathcal{C})}(\unitobj, G(\unitobj)) \cong \Hom_{\mathcal{D}}(\unitobj, \unitobj) \ne 0$.
  We decompose the object $G(\unitobj)$ into a direct sum of indecomposable injective objects of $\Ind(\mathcal{C}')$, as $G(\unitobj) = \bigoplus_{i \in I} E_i$.
  By the assumption that $\mathcal{C}'$ is Frobenius, each $E_i$ belongs to $\mathcal{C}'$.
  Thus, in particular, there are an object $V \in \mathcal{C}'$ and an epimorphism $\pi': G(\unitobj) \to V$ in $\Ind(\mathcal{C}')$.

  By the assumption that $\mathcal{D}$ is Frobenius, there exists a projective cover $\pi: P \to \unitobj$ of $\unitobj$ in $\mathcal{D}$. Furthermore, $P$ is an injective hull of a simple object of $\mathcal{D}$. Since there is an isomorphism $\Hom_{\Ind(\mathcal{C})}(-, G(P)) \cong \Hom_{\mathcal{D}}(F(-), P)$ of functors from $\mathcal{C}$ to $\Vect$, and since $F$ is exact, $E := G(P)$ is an injective object of $\Ind(\mathcal{C})$.
  Since $G$ is assumed to be exact, $G(\pi) : G(P) \to G(\unitobj)$ is an epimorphism in $\Ind(\mathcal{C})$.
  Now we have an epimorphism $\pi' \circ G(\pi) : E \to V$ in $\Ind(\mathcal{C})$ with $V \in \mathcal{C}$ and $E \in \Ind(\mathcal{C})$ injective. By Lemma \ref{lem:Frobenius-criterions}, $\mathcal{C}$ is Frobenius. The proof is completed.
\end{proof}

\section{Applications}
\label{sec:applications}

\subsection{Comodules over braided Hopf algebras}
\label{subsec:comod-over-hopf}

Let $\mathcal{C}$ be a tensor category, and let $H$ be a coalgebra in $\Ind(\mathcal{C})$.
We denote by $\Ind(\mathcal{C})^H$ and $\mathcal{C}^H$ the category of right $H$-comodules in $\Ind(\mathcal{C})$ and its full subcategory consisting of all comodules whose underlying object is finite.
According to Lyubashenko \cite[Section 2.3]{MR1625495}, we may identify $\Ind(\mathcal{C})^H$ with the ind-completion of $\mathcal{C}^H$ (see Appendix \ref{sec:fundamental-theorem}).

Suppose that $\mathcal{C}$ is semisimple and braided, and let $\{ V_{\lambda} \}_{\lambda \in \Lambda}$ be a complete set of representatives of isomorphism classes of simple objects of $\mathcal{C}$. For simplicity of notation, we assume that there is a special index $0 \in \Lambda$ such that $V_0 = \unitobj$.
As $\mathcal{C}$ is assumed to have a braiding, the notion of a Hopf algebra in $\Ind(\mathcal{C})$ makes sense.
Let $H$ be a Hopf algebra in $\Ind(\mathcal{C})$ with comultiplication $\Delta$, counit $\varepsilon$ and antipode $S$.
By the semisimplicity of $\mathcal{C}$, $H$ is decomposed as $H = \bigoplus_{\lambda \in \Lambda} H_{\lambda}$, where $H_{\lambda}$ is the sum of all subobjects of $H$ isomorphic to $V_{\lambda}$.
Let $J$ be the ideal of $H$ generated by $\bigoplus_{\lambda \in \Lambda \setminus \{ 0 \}} H_{\lambda}$. Then $J$ is a Hopf ideal of $H$, and hence the quotient $H / J$ is a Hopf algebra in $\Vect$ (that is, a Hopf algebra in an ordinary sense), where $\Vect$ is viewed as a full subcategory of $\Ind(\mathcal{C})$ via $\bfk \mapsto \unitobj$.

\begin{definition}[{\cite[Definition 2.6]{2019arXiv190911240V}}]
  Given a Hopf algebra $H$ in $\Ind(\mathcal{C})$, we denote by $\overline{H}$ the Hopf algebra $H/J$ mentioned in the above and call it the {\em underlying ordinary Hopf algebra} of $H$.
\end{definition}

In the case where $\mathcal{C}$ is the category of supervector spaces and $H$ is supercommutative, then $\overline{H}$ is the coordinate algebra of the even part of the affine supergroup scheme $\mathrm{Spec}(H)$.
The following theorem is motivated by the result of \cite{MR4375528} on integrals on affine algebraic supergroup schemes (see Example \ref{ex:Masuoka-Shibata-Shimada} for the discussion).

\begin{theorem}
  \label{thm:cointegral}
  Let $H$ be a Hopf algebra in $\Ind(\mathcal{C})$ with invertible antipode.
  Suppose that there is an object $W \in \mathcal{C}$ and an isomorphism $    H \cong \overline{H} \otimes W$ of left $\overline{H}$-comodules in $\Ind(\mathcal{C})$, where $H$ is viewed as a left $\overline{H}$-comodule through the quotient morphism $H \to \overline{H}$. Then the following are equivalent:
  \begin{enumerate}
  \item The underlying Hopf algebra $\overline{H}$ is co-Frobenius.
  \item The category of right $H$-comodules in $\mathcal{C}$ is a Frobenius tensor category.
  \end{enumerate}
\end{theorem}
\begin{proof}
  By the assumption that the antipode of $H$ is invertible, both $\mathcal{C}^H$ and $\mathcal{C}^{\overline{H}}$ are tensor categories. Let $q: H \to \overline{H}$ be the quotient morphism. We consider the functor
  \begin{equation*}
    F: \mathcal{C}^H \to \mathcal{C}^{\overline{H}},
    \quad (M, \delta_M) \mapsto (M, (\id_M \otimes q) \delta_M)
  \end{equation*}
  induced by $q$. The functor $\Ind(F)$ has a right adjoint $G: \Ind(\mathcal{C})^{\overline{H}} \to \Ind(\mathcal{C})^H$ given by the cotensor product $G(M) = M \mathbin{\square}_{\overline{H}} H$ for $M \in \Ind(\mathcal{C})^{\overline{H}}$.
  Since we have
  \begin{equation*}
    G(M) \cong M \mathbin{\square}_{\overline{H}} (\overline{H} \otimes W) \cong M \otimes W
    \quad (M \in \Ind(\mathcal{C}^{\overline{H}}))
  \end{equation*}
  as objects of $\Ind(\mathcal{C})$, we have $G(\mathcal{C}^{\overline{H}}) \subset \mathcal{C}^H$ and that $G$ is exact.
  Now we can apply Theorem~\ref{thm:Frobenius-and-tensor-functors} to show that (2) is equivalent to that the tensor category $\mathcal{C}^{\overline{H}}$ is Frobenius.
  Since $\mathcal{C}$ is semisimple and $\overline{H}$ belongs to $\Vect$, we have $\mathcal{C}^{\overline{H}} \approx (\Mod^{\overline{H}}_{\fd})^{\oplus \Lambda}$ as linear categories.
  Thus the tensor category $\mathcal{C}^{\overline{H}}$ is Frobenius if and only if so is $\Mod^{\overline{H}}_{\fd}$.
  By Lemma~\ref{lem:Frobenius-criterions-0} and \cite[Chapter 5]{MR1786197}, the latter condition is equivalent to that $\overline{H}$ is co-Frobenius. The proof is done.
\end{proof}

\begin{example}
  [Integrals on affine algebraic supergroup schemes \cite{MR4375528}]
  \label{ex:Masuoka-Shibata-Shimada}
  We assume that the base field $\bfk$ is of characteristic $p \ne 2$.
  We take $\mathcal{C}$ to be the category of supervector spaces.
  Let $G$ be an affine algebraic supergroup scheme represented by a Hopf superalgebra $H := \mathcal{O}(G)$.
  We denote by $\Rep(G)$ the category of finite-dimensional supercomodules over $H$.
  Masuoka \cite[Theorem 4.5]{MR2163412} showed that there is a finite-dimensional vector space $V$ and an isomorphism $H \cong \overline{H} \otimes_{\bfk} \bigwedge V$ of left $\overline{H}$-comodules, where $\bigwedge V$ means the exterior algebra of $V$.
  Thus, by Theorem \ref{thm:cointegral}, $\Rep(G)$ is Frobenius if and only if the Hopf algebra $\mathcal{O}(G_{\mathrm{ev}})$ is co-Frobenius, where $G_{\mathrm{ev}}$ is the even part of $G$ ({\it cf}. \cite[Theorem 3.7]{MR4375528}).
\end{example}

\begin{example}
  [Affine algebraic group schemes in the Verlinde category \cite{MR3556434,2019arXiv190911240V,2022arXiv220303158V}]
  \label{ex:Venkatesh}
  \newcommand{\Ver}{\mathrm{Ver}}
  We assume that the base field $\bfk$ is an algebraically closed field of characteristic $p > 2$.
  Then the symmetric finite tensor category $\mathcal{C} := \Ver_p$, called the Verlinde category, is defined as the semisimplification of the category of finite-dimensional representations of the cyclic group of order $p$.
  Commutative Hopf algebras in $\Ind(\mathcal{C})$ were studied by Venkatesh in \cite{MR3556434,2019arXiv190911240V,2022arXiv220303158V}.
  Let $H$ be a finitely generated commutative Hopf algebra in $\Ind(\mathcal{C})$.
  According to \cite[Proposition 3.1]{MR3556434} and \cite[Lemma 7.15]{2019arXiv190911240V}, there is an object $W \in \mathcal{C}$ and an isomorphism $H \cong \overline{H} \otimes W$ of left $\overline{H}$-comodules.
  Thus, by Theorem \ref{thm:cointegral}, the tensor category $\mathcal{C}^H$ is Frobenius if and only if the Hopf algebra $\overline{H}$ is co-Frobenius.
\end{example}

\subsection{De-equivariantization by an affine group scheme}
\label{subsec:de-equiv}

Given a category on which a group acts, its equivariantization is defined as the category of `fixed points' of the action of the group \cite{MR2609644}. The de-equivariantization is, in a sense, the inverse of this procedure.
As pointed out in \cite{MR2863377,MR3161401}, the de-equivariantization by a finite group can be understood in terms of exact sequences of tensor categories.
With a view to the results on the de-equivariantization of Hopf algebras by affine group schemes \cite{MR3160718}, we slightly extend the argument of \cite{MR3161401} and discuss how our results can be applied to the de-equivariantization by an affine group scheme.

We first prepare some lemmas to introduce the de-equivariantization.
Let $\mathcal{C}$ be a tensor category.
Given an algebra $A$ in $\Ind(\mathcal{C})$, we denote by $\Ind(\mathcal{C})_A$ the category of right $A$-modules in $\Ind(\mathcal{C})$.
Since $\Ind(\mathcal{C})$ is a Grothendieck category and $(-) \otimes A$ is a linear exact cocontinuous monad on it, $\Ind(\mathcal{C})_A$ is a Grothendieck category.

We say that an object $X$ of a category $\mathcal{A}$ is {\em finitely presented} (respectively, {\em finitely generated}) if the functor $\Hom_{\mathcal{A}}(X, -) : \mathcal{A} \to \Sets$ preserves filtered colimits (respectively, filtered colimits of monomorphisms).
When the category $\mathcal{A}$ is Grothendieck, it is known that the following assertions are equivalent for an object $X \in \mathcal{A}$ ({\it cf}. \cite[Chapter V, Proposition 3.2]{MR0389953}):
\begin{enumerate}
\item $X$ is finitely generated.
\item For every filtered family $\{ X_i \}_{i \in I}$ of subobjects of $X$ such that $X = \varinjlim_{i \in I} X_i$, there is an element $i_0 \in I$ such that $X_{i_0} = X$.
\item For every family $\{ X_\lambda \}_{\lambda \in \Lambda}$ of subobjects of $X$ such that $X = \sum_{\lambda \in \Lambda} X_\lambda$, there is a finite subset $\{ \lambda_1, \cdots, \lambda_n \}$ of $\Lambda$ such that $X = \sum_{i = 1}^n X_{\lambda_i}$.
\end{enumerate}
We also note that the following assertions are equivalent for an object $X \in \mathcal{A}$ ({\it cf}. \cite[Chapter V, Proposition 3.4]{MR0389953}):
\begin{enumerate}
\item $X$ is finitely presented.
\item For any epimorphism $f: W \to X$ in $\mathcal{A}$ with $W$ a finitely generated object of $\mathcal{A}$, the kernel of $f$ is finitely generated.
\end{enumerate}

Finitely generated objects of $\Ind(\mathcal{C})_A$ are characterized as follows:

\begin{lemma}
  \label{lem:fg-objects-in-Ind-C-A}
  An object of $\Ind(\mathcal{C})_A$ is finitely generated if and only if it is a quotient of a right $A$-module of the form $F_A(X)$ for some object $X \in \mathcal{C}$, where $F_A$ is the free right $A$-module functor
  \begin{equation}
    \label{eq:Ind-C-A-free-functor}
    F_A : \Ind(\mathcal{C}) \to \Ind(\mathcal{C})_A,
    \quad X \mapsto X \otimes A.
  \end{equation}
\end{lemma}
\begin{proof}
  Suppose that a right $A$-module $M$ is finitely generated. Let $\{ X_{\lambda} \}_{\lambda \in \Lambda}$ be the set of all finite subobjects of $M$ (as an object of $\Ind(\mathcal{C})$). For each $\lambda$, we define $M_{\lambda}$ to be the image of $X_{\lambda} \otimes A$ under the action $M \otimes A \to M$. Then $\{ M_{\lambda} \}_{\lambda \in \Lambda}$ is a family of submodule of $M$ such that $M = \sum_{\lambda \in \Lambda} M_{\lambda}$.
  By the assumption that $M$ is finitely generated, there is a finite subset $\{ \lambda_1, \cdots, \lambda_n \}$ of $\Lambda$ such that $M = \sum_{i = 1}^n M_{\lambda_i}$. Hence $M$ is a quotient of $F_A(X_{\lambda_1} \oplus \dotsb \oplus X_{\lambda_n})$.

  We show the converse. Since a quotient of a finitely generated object is finitely generated, it suffices to show that a right $A$-module of the form $F_A(V)$, $V \in \mathcal{C}$, is finitely generated. Let $V$ be an object of $\mathcal{C}$ (we note that then $V$ is a finitely presented object of $\Ind(\mathcal{C})$). Let $U: \Ind(\mathcal{C})_A \to \Ind(\mathcal{C})$ be the forgetful functor.
  Then there is an isomorphism $\Hom_A(F_A(V), -) \cong \Hom_{\Ind(\mathcal{C})}(V, -) \circ U$ of functors.
  Since both $\Hom_{\Ind(\mathcal{C})}(V, -)$ and $U$ preserve filtered colimits, $F_A(V)$ is finitely presented. In particular, it is finitely generated. The proof is done.
\end{proof}

Let $\mathrm{mod}_{\mathcal{C}}(A)$ denote the full subcategory of $\Ind(\mathcal{C})_A$ consisting of finitely generated objects.
The category $\mathrm{mod}_{\mathcal{C}}(A)$ is closed under cokernels, however, may not be closed under kernels.
The following lemma says that $\mathrm{mod}_{\mathcal{C}}(A)$ has kernels if and only if $\mathrm{mod}_{\mathcal{C}}(A)$ is closed under kernels.

\begin{lemma}
  \label{lem:kernel-in-fgmod}
  Let $f: M \to N$ be a morphism in $\mathrm{mod}_{\mathcal{C}}(A)$.
  Then the kernel of $f$ exists in $\mathrm{mod}_{\mathcal{C}}(A)$ if and only if the kernel of $f$ in $\Ind(\mathcal{C})_A$ is finitely generated.
\end{lemma}
\begin{proof}
  The `if' part is clear.
  The converse is proved as follows:
  We suppose that the morphism $f$ has a kernel in $\mathrm{mod}_{\mathcal{C}}(A)$ and denote it by $k: K \to M$.
  Let $g : X \to M$ be a morphism in $\Ind(\mathcal{C})_A$ such that $f \circ g = 0$.
  Since the underlying object of the right $A$-module $X$ is an inductive limit of objects of $\mathcal{C}$, we can write $X$ as an inductive limit of finitely generated subobjects of $X$, as $X = \varinjlim_{\lambda \in \Lambda} X_{\lambda}$. For each $\lambda$, we denote by $i_{\lambda} : X_{\lambda} \to X$ the canonical morphism.
  By the universal property of the kernel, there exists a unique morphism $\phi_{\lambda} : X_{\lambda} \to K$ of right $A$-modules such that $g \circ i_{\lambda} = k \circ \phi_{\lambda}$.
  There is a morphism $\phi : X \to K$ of right $A$-modules such that $\phi \circ i_{\lambda} = \phi_{\lambda}$ for all $\lambda \in \Lambda$.
  The morphism $\phi$ satisfies $k \circ \phi = g$ and this equation uniquely characterizes $\phi$. The proof is done.
\end{proof}

The Drinfeld center \cite{MR3242743} of a monoidal category $\mathcal{M}$, denoted by $\mathcal{Z}(\mathcal{M})$, is the braided monoidal category whose object is a pair $(M, \sigma_M)$ consisting of an object $M \in \mathcal{M}$ and a natural isomorphism $\sigma_{M,X} : M \otimes X \to X \otimes M$ ($X \in \mathcal{M}$) satisfying a condition like the axiom for braiding. Now let $(A, \sigma_A)$ be a commutative algebra in $\mathcal{Z}(\Ind(\mathcal{C}))$. Then a right $A$-module is made into an $A$-bimodule through the natural isomorphism $\sigma$. Moreover, $\Ind(\mathcal{C})_A$ is a monoidal category with respect to the tensor product $\otimes_A$ over $A$. As desired, we have:

\begin{lemma}
  \label{lem:fg-objects-in-Ind-C-A-tensor}
  The full subcategory $\mathrm{mod}_{\mathcal{C}}(A)$ is closed under $\otimes_A$.
\end{lemma}
\begin{proof}
  Let $M$ and $N$ be finitely generated objects of $\Ind(\mathcal{C})_A$.
  By Lemma \ref{lem:fg-objects-in-Ind-C-A}, $M$ and $N$ are quotients of objects of the form $F_A(X)$ and $F_A(Y)$, respectively, for some objects $X$ and $Y$ of $\mathcal{C}$. By the right exactness of $\otimes_A$, the tensor product $M \otimes_A N$ is a quotient of the right $A$-module $F_A(X) \otimes_A F_A(Y)$, which is isomorphic to $F_A(X \otimes Y)$.
  Hence, by Lemma \ref{lem:fg-objects-in-Ind-C-A}, $M \otimes_A N$ is finitely generated. The proof is done.
\end{proof}

Let $\mathcal{C}$ be a tensor category, and let $\mathcal{B}$ be a braided tensor category.
A {\em central inclusion functor} for $\mathcal{B}$ into $\mathcal{C}$ is a fully faithful tensor functor $\iota: \mathcal{B} \to \mathcal{C}$ together with a tensor functor $F: \mathcal{B} \to \mathcal{Z}(\mathcal{C})$ such that $F$ preserves the braiding and the equation $U \circ F = \iota$ holds, where $U: \mathcal{Z}(\mathcal{C}) \to \mathcal{C}$ is the forgetful functor.
We remark:

\begin{lemma}
  \label{lem:Ind-lifting}
  Let $\mathcal{C}$ be a tensor category, and let $\mathcal{B}$ be a braided tensor category.
  Then any braided tensor functor $F: \mathcal{B} \to \mathcal{Z}(\mathcal{C})$ extends to a `braided tensor functor' from $\Ind(\mathcal{B})$ to $\mathcal{Z}(\Ind(\mathcal{C}))$.
\end{lemma}

Here, despite that $\Ind(\mathcal{B})$ and $\mathcal{Z}(\Ind(\mathcal{C}))$ are not tensor categories, we use the word `braided tensor functor' to mean a linear functor from $\Ind(\mathcal{B})$ to $\mathcal{Z}(\Ind(\mathcal{C}))$ preserving the tensor product and the braiding.

\begin{proof}
  Given an object $V \in \mathcal{B}$, we write $F(V) \in \mathcal{Z}(\mathcal{C})$ as $F(V) = (f(V), \sigma_{f(V)})$.
  Let $X$ and $M$ be objects of $\Ind(\mathcal{B})$ and $\Ind(\mathcal{C})$, respectively, and express them as filtered colimits of finite objects, as $X = \varinjlim_{\lambda \in \Lambda} X_{\lambda}$ and $M = \varinjlim_{\gamma \in \Gamma} M_{\gamma}$.
  We set $\tilde{f}(X) = \varinjlim_{\lambda \in \Lambda} f(X_{\lambda})$ and let $\tilde{\sigma}_{X,M} : \tilde{f}(X) \otimes M \to M \otimes \tilde{f}(X)$ be the morphism in $\Ind(\mathcal{C})$ induced by the family
  \begin{equation*}
    \sigma_{f(X_{\lambda}), M_{\gamma}} : f(X_{\lambda}) \otimes M_{\gamma} \to M_{\gamma} \otimes f(X_{\lambda})
    \quad (\lambda \in \Lambda, \gamma \in \Gamma)
  \end{equation*}
  of morphisms in $\mathcal{C}$.
  One can check that the assignment $X \mapsto (\tilde{f}(X), \tilde{\sigma}_X)$ gives rise to a `braided tensor functor' from $\Ind(\mathcal{B})$ to $\mathcal{Z}(\Ind(\mathcal{C}))$, as desired.
\end{proof}

Let $G$ be an affine group scheme represented by $\mathcal{O}(G)$.
As is well-known, $\Rep(G) := \Mod^{\mathcal{O}(G)}_{\fd}$ is a symmetric tensor category and the algebra $\mathcal{O}(G)$ is a commutative algebra in $\Ind(\Rep(G))$ in a natural way.
Now let $\mathcal{C}$ be a tensor category endowed with a central inclusion functor $\iota: \Rep(G) \to \mathcal{C}$. Lemma \ref{lem:Ind-lifting} allows us to view $\mathcal{O}(G)$ as a commutative algebra in $\mathcal{Z}(\Ind(\mathcal{C}))$. Thus the monoidal category $\Ind(\mathcal{C})_{\mathcal{O}(G)}$ is defined.

\begin{theorem}
  \label{thm:de-equivari}
  Let $G$, $\mathcal{C}$ and $\iota: \Rep(G) \to \mathcal{C}$ be as above.
  Suppose that the full subcategory $\mathcal{D} := \mathrm{mod}_{\mathcal{C}}(\mathcal{O}(G))$ of finitely generated objects of $\Ind(\mathcal{C})_{\mathcal{O}(G)}$ is a tensor category with respect to $\otimes_{\mathcal{O}(G)}$.
  Then the following hold:
  \begin{enumerate}
  \item $\mathcal{D}$ coincides with the full subcategory of finitely presented objects of $\Ind(\mathcal{C})_{\mathcal{O}(G)}$.
  \item There is an equivalence $\Ind(\mathcal{C})_{\mathcal{O}(G)} \approx \Ind(\mathcal{D})$ of linear monoidal categories.
  \item There is an exact sequence $\Rep(G) \to \mathcal{C} \to \mathcal{D}$ of tensor categories.
  \item The tensor category $\mathcal{C}$ is Frobenius if and only if the Hopf algebra $\mathcal{O}(G)$ is co-Frobenius and the tensor category $\mathcal{D}$ is Frobenius.
  \end{enumerate}
\end{theorem}

In view of the well-studied case where $G$ is a finite group, we may call $\mathcal{D}$ the {\em de-equivariantization} of $\mathcal{C}$ by $G$ at least when $\mathcal{D}$ satisfies the assumption of this theorem.
Admittedly, in general, it is not easy to check whether $\mathcal{D}$ is a tensor category.
As discussed in \cite[Section 6]{MR4227163}, when $\mathcal{C}$ is the category of comodules over a Hopf algebra $H$, the rigidity of $\mathcal{D}$ closely relates to the faithfully flatness of $H$ over a Hopf subalgebra.

Unfortunately, we do not yet know any significant examples where the Frobenius property of a tensor category is proved by using Theorem \ref{thm:de-equivari}.
Nevertheless, we think this theorem is noteworthy in view of future applications.
As supporting evidence, we point out that the idea of the de-equivariantization by an affine algebraic group scheme is used to construct new classes of modular tensor categories related to logarithmic conformal field theories in \cite{2018arXiv180902116G,MR4227163}.

\begin{proof}[Proof of Theorem~\ref{thm:de-equivari}]
  We write $A = \mathcal{O}(G)$.

  \medskip
  (1) We recall that a tensor category is, by definition, a locally finite abelian category.
  By Lemma~\ref{lem:kernel-in-fgmod} and the assumption that $\mathcal{D}$ is a tensor category, we see that the subcategory $\mathcal{D}$ of $\Ind(\mathcal{C})_A$ is closed under kernels. Hence $\mathcal{D}$ coincides with the full subcategory of finitely presented objects of $\Ind(\mathcal{C})_A$.

  \medskip
  (2) Since $\Ind(\mathcal{C})_A$ is a Grothendieck category such that every object of $\Ind(\mathcal{C})_A$ is a filtered colimit of finitely presented objects, there is an equivalence $\Ind(\mathcal{C})_A \approx \Ind(\mathcal{D})$ of categories by \cite[Corollary 6.3.5]{MR2182076}.
  It is easy to see that this is an equivalence of linear monoidal categories.

  \medskip
  (3) By Lemma \ref{lem:fg-objects-in-Ind-C-A}, we see that the free right module functor \eqref{eq:Ind-C-A-free-functor} restricts to a tensor functor from $\mathcal{C}$ to $\mathcal{D}$, which we denote by $F : \mathcal{C} \to \mathcal{D}$.
  Lemma \ref{lem:fg-objects-in-Ind-C-A} also implies that $F$ is dominant.

  We show that $\Ker(F)$ coincides with $\Rep(G)$ viewed as a full subcategory of $\mathcal{C}$ through $\iota$.
  The fundamental theorem for Hopf modules \cite{MR1243637} implies that $\Rep(G)$ is contained in $\Ker(F)$.
  Now we suppose that an object $X \in \mathcal{C}$ belongs to $\Ker(F)$.
  Then we have $X \otimes A = F(A) \cong A^{\oplus n}$ for some non-negative integer $n$, and thus $X$ is a subobject of $A^{\oplus n} \in \Ind(\mathcal{C})$.
  This implies that $X$ belongs to $\Rep(G)$.
  Hence we have proved $\Ker(F) = \Rep(G)$.

  To complete the proof of Part (3), we shall verify the normality of $F$.
  We identify $\Ind(\mathcal{D})$ with $\Ind(\mathcal{C})_A$ by Part (1) and let $U: \Ind(\mathcal{D}) \to \Ind(\mathcal{C})$ be the functor forgetting the $A$-module structure.
  It is easy to see that $U$ is an ind-adjoint of $F$.
  The unit object of $\mathcal{D}$ is $A$, which is mapped to $A \in \Ind(\Rep(G))$ by the functor $U$. Thus, by Lemma \ref{lem:normality-BN11-Prop-3-5}, $F$ is normal.

  \medskip
  (4) The functor $U$ is obviously exact. Thus, by Theorem~\ref{thm:Frobenius-closed-under-exact-seq}, the tensor category $\mathcal{C}$ is Frobenius if and only if so are $\Rep(G)$ and $\mathcal{D}$. The proof of this theorem is completed by noting the fact that $\Rep(G)$ is Frobenius if and only if $\mathcal{O}(G)$ is co-Frobenius.
\end{proof}

\begin{example}
  Let $H$ be a Hopf algebra with bijective antipode, and let $K$ be a {\em braided central Hopf subalgebra} of $H$ \cite[Definition 3.1]{MR3160718}, that is, a Hopf subalgebra $K \subset H$ together with a bilinear map $r: H \times K \to \bfk$ satisfying certain axioms similar to those of universal R-forms.
  Then $K$ is commutative and thus we may assume that $K = \mathcal{O}(G)$ for some affine group scheme $G$.
  We set $\mathcal{C} = {}^H\Mod_{\fd}$.
  By the argument of \cite{MR3160718}, we have a central inclusion functor $\Rep(G^{\op}) \to \mathcal{C}$ induced by $r$ (here the source of the functor is set to $\Rep(G^{\op})$, which is identified with ${}^{\mathcal{O}(G)}\Mod_{\fd}$, to be consistent with the setting of \cite{MR3160718} and our notation).

  Now we suppose that there is a convolution-invertible $\mathcal{O}(G)$-linear map $\pi: H \to \mathcal{O}(G)$ preserving the unit and the counit.
  We consider the quotient coalgebra $Q := H/\mathcal{O}(G)^{+}H$, where $\mathcal{O}(G)^{+}$ is the kernel of the counit of $\mathcal{O}(G)$.
  The main result of \cite{MR3160718} states that $Q$ has a structure of a coquasi-bialgebra given explicitly in terms of $r$ and $\pi$, and the category ${}^Q\Mod$ is equivalent to the category of $\mathcal{O}(G)$-modules in $\Ind(\mathcal{C})$ as a linear monoidal category.
  Theorem \ref{thm:de-equivari} yields the following consequence:
  Suppose that the monoidal category ${}^Q\Mod_{\fd}$ is rigid.
  Then $Q$ is a QcF coalgebra if and only if both $\mathcal{O}(G)$ and $H$ are co-Frobenius.
\end{example}

\appendix
\section{The fundamental theorem for comodules}
\label{sec:fundamental-theorem}

Let $\mathcal{V}$ be a tensor category over a field $\bfk$, and let $\mathbf{C}$ be a coalgebra in $\Ind(\mathcal{V})$ with underlying object $C \in \Ind(\mathcal{V})$, comultiplication $\Delta_C$ and counit $\varepsilon_C$.
As in Subsection \ref{subsec:comod-over-hopf}, we denote by $\Ind(\mathcal{V})^{\mathbf{C}}$ and $\mathcal{V}^{\mathbf{C}}$ the category of right $\mathbf{C}$-comodules in $\Ind(\mathcal{V})$ and its full subcategory consisting of objects whose underlying object belongs to $\mathcal{V}$.
The aim of this appendix is to prove the following theorem:

\begin{theorem}
  \label{thm:appendix-fund-thm-comodules}
  $\Ind(\mathcal{V}^{\mathbf{C}}) \approx \Ind(\mathcal{V})^{\mathbf{C}}$.
\end{theorem}

In the case where $\mathcal{V}$ is the category of finite-dimensional vector spaces over $\bfk$, this theorem follows from the fundamental theorem for comodules.
The general case was considered by Lyubashenko in \cite[Section 2.3]{MR1625495}.
Strictly speaking, in \cite{MR1625495}, Theorem \ref{thm:appendix-fund-thm-comodules} is proved in a general setting of what Lyubashenko called squared coalgebras.
Since it may be hard to specialize the proof given in \cite{MR1625495} in our setting, we include a direct proof of Theorem \ref{thm:appendix-fund-thm-comodules} for the reader's convenience.

We first note the following easy observation:

\begin{lemma}
  \label{lem:appendix-finite-sub-tensor}
  Let $X$ and $Y$ be objects of $\Ind(\mathcal{V})$, and let $V$ be a subobject of $X \otimes Y$ of finite length. Then there are subobjects $X' \subset X$ and $Y' \subset Y$ of finite lengths such that $V \subset X' \otimes Y'$.
\end{lemma}
\begin{proof}
  \newcommand{\Mu}{\mathrm{M}}
  We write $X$ and $Y$ as the sum of subobjects of finite lengths, as $X = \sum_{\lambda \in \Lambda} X_{\lambda}$ and $Y = \sum_{\mu \in \Mu} Y_{\mu}$.
  Since the tensor product of $\Ind(\mathcal{V})$ is cocontinuous, we have $X \otimes Y = \sum_{\lambda \in \Lambda, \mu \in \Mu} X_{\lambda} \otimes Y_{\mu}$. By the finiteness of $V$, there are finitely many elements $\lambda_1, \cdots, \lambda_{r} \in \Lambda$ and $\mu_1, \cdots, \mu_s \in \Mu$ such that $V \subset \sum_{i = 1}^r \sum_{j = 1}^s X_{\lambda_i} \otimes Y_{\mu_j}$. Now we set $X' = \sum_{i = 1}^r X_{\lambda_i}$ and $Y' = \sum_{j = 1}^s Y_{\mu_j}$. Then $V \subset X' \otimes Y'$, as desired.
\end{proof}

For objects $X, Y \in \Ind(\mathcal{V})$ and $D \in \mathcal{V}$, there is the bijection
\begin{align*}
  (-)^{\sharp} : \Hom_{\Ind(\mathcal{V})}(X, Y \otimes D) & \to
  \Hom_{\Ind(\mathcal{V})}(X \otimes {}^{\vee}\!D, Y), \\
  f & \mapsto f^{\sharp} := (\id_Y \otimes \eval'_{D})(f \otimes \id_{{}^{\vee}\!D}),
\end{align*}
where $\eval'_D : D \otimes {}^{\vee}\!D \to \unitobj$ is the evaluation morphism.

\begin{lemma}
  \label{lem:appendix-delta-sharp}
  For a morphism $f: X \to Y \otimes D$ in $\Ind(\mathcal{V})$, we have
  \begin{equation*}
    \Img(f) \subset \Img(f^{\sharp}) \otimes D, \quad
    \Ker(f) \otimes {}^{\vee}\!D \subset \Ker(f^{\sharp}).
  \end{equation*}
\end{lemma}
\begin{proof}
  Let $\pi : Y \to \Coker(f^{\sharp})$ be the cokernel of $f^{\sharp}$. Then we have
  \begin{equation*}
    (\pi \otimes \id_{D}) f
    = (\pi \otimes \id_{D}) (f^{\sharp} \otimes \id_{D}) (\id_X \otimes \coev'_{D}) = 0,
  \end{equation*}
  where $\coev'_{D} : \unitobj \to {}^{\vee}\!D \otimes D$ is the coevaluation. Hence,
  \begin{equation*}
    \Img(f) \subset \Ker(\pi \otimes \id_{D})
    = \Ker(\pi) \otimes D
    = \Img(f^{\sharp}) \otimes D,
  \end{equation*}
  where we have used the exactness of $\otimes$.
  $\Ker(f) \otimes {}^{\vee}\!D \subset \Ker(f^{\sharp})$ is proved by the dual argument.
\end{proof}

Given a morphism $f: X \to Y$ in $\Ind(\mathcal{V})$ and a subobject $V \subset X$, we define $f(V)$ to be the image of $f$ restricted to the subobject $V \subset X$.
Now we prove the following analogue of the fundamental theorem for comodules:

\begin{lemma}
  \label{lem:appendix-A-4}
  Let $\mathbf{M}$ be a right $\mathbf{C}$-comodule in $\Ind(\mathcal{V})$ with underlying object $M \in \Ind(\mathcal{V})$ and coaction $\delta_M : M \to M \otimes C$, and let $V$ be a subobject of $M$ of finite length. Then there is a subobject $N$ of $M$ of finite length such that $V \subset N$ and $\delta_M(N) \subset N \otimes C$.
\end{lemma}
\begin{proof}
  By Lemma \ref{lem:appendix-finite-sub-tensor}, there is a subobject $D_1 \subset C$ of finite length such that $\delta_M(V) \subset M \otimes D_1$. Let $\delta : V \to M \otimes D_1$ be the morphism induced by $\delta_M : M \to M \otimes C$, and set $N = \Img(\delta^{\sharp} : V \otimes {}^{\vee}\!D_1 \to M)$. Then $N$ is a subobject of $M$ of finite length as the image of an object of finite length.
  Moreover, we have
  \begin{equation*}
    \delta(V) \subset \Img(\delta^{\sharp}) \otimes D_1 = N \otimes D_1
  \end{equation*}
  by Lemma \ref{lem:appendix-delta-sharp}. Thus, from now on, we regard $\delta$ as a morphism from $V$ to $N \otimes D_1$. We note that, by the counit axiom and Lemma \ref{lem:appendix-delta-sharp}, we have
  \begin{equation*}
    V
    = (\id_M \otimes \varepsilon_C|_{D_1}) \delta(V)
    \subset (\id_M \otimes \varepsilon_C|_{D_1}) (\Img(\delta^{\sharp}) \otimes D_1)
    = \Img(\delta^{\sharp}) = N.
  \end{equation*}
  To complete the proof, we show that $\delta_M(N) \subset N \otimes C$.
  By Lemma \ref{lem:appendix-finite-sub-tensor}, there is a subobject $D_2 \subset C$ of finite length such that $\Delta_C(D_1) \subset C \otimes D_2$. By taking the image of these objects under $\varepsilon_C \otimes \id_C$, we obtain $D_1 \subset D_2$. Now let $\Delta: D_1 \to C \otimes D_2$ be the morphism induced by the comultiplication $\Delta_C : C \to C \otimes C$, and let $i: N \to M$ and $j : D_1 \to D_2$ be the inclusion morphisms. Then we have
  \begin{equation*}
    (\id_M \otimes \id_C \otimes j)(\delta_M i \otimes \id_{D_1}) \delta
    = (i \otimes \Delta) \delta
  \end{equation*}
  by the coassociativity of the coaction.
  By applying the operator
  \begin{equation*}
    (-)^{\sharp}: \Hom_{\Ind(\mathcal{V})}(V, M \otimes C \otimes D_2)
    \to \Hom_{\Ind(\mathcal{V})}(V \otimes {}^{\vee}\!D_2, M \otimes C)
  \end{equation*}
  to the both sides and regarding $\delta^{\sharp}$ as a morphism from $M \otimes {}^{\vee}\!D_1$ to $N$, we obtain
  \begin{equation*}
    \delta_M i \delta^{\sharp} (\id_V \otimes {}^{\vee}\!j)
    = (\id_M \otimes \eval'_{D_2}) (i \otimes \id_C \otimes \Delta \otimes \id_{{}^{\vee}\!D_2}) (\delta \otimes \id_{{}^{\vee}\!D_2}).
  \end{equation*}
  The morphism ${}^{\vee}\!j$ is an epimorphism as a dual of the monomorphism $j$.
  Since $\delta^{\sharp} (\id_V \otimes {}^{\vee}\!j) : V \otimes {}^{\vee}D_2 \to N$ is an epimorphism, the image of the left hand side is $\delta_M(N)$.
  It is obvious that the image of the right hand side is contained in $N \otimes C$. Thus $\delta_M(N) \subset N \otimes C$. The proof is done.
\end{proof}

\begin{proof}[Proof of Theorem \ref{thm:appendix-fund-thm-comodules}]
  By Lemma \ref{lem:appendix-A-4}, we see that the class of objects of $\mathcal{V}^{\mathbf{C}}$, the class of finitely generated objects of $\Ind(\mathcal{V})^{\mathbf{C}}$ and the class of finitely presented objects of $\Ind(\mathcal{V})^{\mathbf{C}}$ coincide.
  Moreover, every object of $\Ind(\mathcal{V})^{\mathbf{C}}$ is a filtered colimit of objects of $\mathcal{V}^{\mathbf{C}}$.
  Thus, by \cite[Corollary 6.3.5]{MR2182076}, there is an equivalence $\Ind(\mathcal{V})^{\mathbf{C}} \approx \Ind(\mathcal{V}^{\mathbf{C}})$ of categories.
\end{proof}


\begin{thebibliography}{DGNO10}

\bibitem[AC13]{MR3032811}
N. Andruskiewitsch and J. Cuadra.
\newblock On the structure of (co-{F}robenius) {H}opf algebras.
\newblock {\em J. Noncommut. Geom.}, 7(1):83--104, 2013.

\bibitem[ACE15]{MR3410615}
N. Andruskiewitsch, J. Cuadra, and P. Etingof.
\newblock On two finiteness conditions for {H}opf algebras with nonzero
  integral.
\newblock {\em Ann. Sc. Norm. Super. Pisa Cl. Sci. (5)}, 14(2):401--440, 2015.

\bibitem[AGP14]{MR3160718}
Iv\'{a}n Angiono, C\'{e}sar Galindo, and Mariana Pereira.
\newblock De-equivariantization of {H}opf algebras.
\newblock {\em Algebr. Represent. Theory}, 17(1):161--180, 2014.

\bibitem[BN11]{MR2863377}
A. Brugui\`eres and Sonia Natale.
\newblock Exact sequences of tensor categories.
\newblock {\em Int. Math. Res. Not. IMRN}, (24):5644--5705, 2011.

\bibitem[BN14]{MR3161401}
A. Brugui\`eres and Sonia Natale.
\newblock Central exact sequences of tensor categories, equivariantization and
  applications.
\newblock {\em J. Math. Soc. Japan}, 66(1):257--287, 2014.

\bibitem[DGNO10]{MR2609644}
V. Drinfeld, S. Gelaki, D. Nikshych, and V. Ostrik.
\newblock On braided fusion categories. {I}.
\newblock {\em Selecta Math. (N.S.)}, 16(1):1--119, 2010.

\bibitem[DNR01]{MR1786197}
S. D\u{a}sc\u{a}lescu, C. N\u{a}st\u{a}sescu, and \c{S}. Raianu.
\newblock {\em Hopf algebras}, volume 235 of {\em Monographs and Textbooks in
  Pure and Applied Mathematics}.
\newblock Marcel Dekker, Inc., New York, 2001.
\newblock An introduction.

\bibitem[EGNO15]{MR3242743}
P. Etingof, S. Gelaki, D. Nikshych, and V. Ostrik.
\newblock {\em Tensor categories}, volume 205 of {\em Mathematical Surveys and
  Monographs}.
\newblock American Mathematical Society, Providence, RI, 2015.

\bibitem[FSS20]{MR4042867}
J. Fuchs, G. Schaumann, and C. Schweigert.
\newblock Eilenberg-{W}atts calculus for finite categories and a bimodule
  {R}adford {$S^4$} theorem.
\newblock {\em Trans. Amer. Math. Soc.}, 373(1):1--40, 2020.

\bibitem[GLO18]{2018arXiv180902116G}
A.~M. {Gainutdinov}, S. {Lentner}, and T. {Ohrmann}.
\newblock {Modularization of small quantum groups}.
\newblock {\tt arXiv:1809.02116}.

\bibitem[Iov14]{MR3150709}
M.~C. Iovanov.
\newblock Generalized {F}robenius algebras and {H}opf algebras.
\newblock {\em Canad. J. Math.}, 66(1):205--240, 2014.

\bibitem[KS06]{MR2182076}
M. Kashiwara and P. Schapira.
\newblock {\em Categories and sheaves}, volume 332 of {\em Grundlehren der
  Mathematischen Wissenschaften [Fundamental Principles of Mathematical
  Sciences]}.
\newblock Springer-Verlag, Berlin, 2006.

\bibitem[Lyu99]{MR1625495}
V.~V. Lyubashenko.
\newblock Squared {H}opf algebras.
\newblock {\em Mem. Amer. Math. Soc.}, 142(677):x+180, 1999.

\bibitem[Mas05]{MR2163412}
A. Masuoka.
\newblock The fundamental correspondences in super affine groups and super
  formal groups.
\newblock {\em J. Pure Appl. Algebra}, 202(1-3):284--312, 2005.

\bibitem[ML98]{MR1712872}
S. Mac~Lane.
\newblock {\em Categories for the working mathematician}, volume~5 of {\em
  Graduate Texts in Mathematics}.
\newblock Springer-Verlag, New York, second edition, 1998.

\bibitem[Mon93]{MR1243637}
S. Montgomery.
\newblock {\em Hopf algebras and their actions on rings}, volume~82 of {\em
  CBMS Regional Conference Series in Mathematics}.
\newblock Published for the Conference Board of the Mathematical Sciences,
  Washington, DC; by the American Mathematical Society, Providence, RI, 1993.

\bibitem[MSS22]{MR4375528}
A. Masuoka, T. Shibata, and Y. Shimada.
\newblock Affine algebraic super-groups with integral.
\newblock {\em Comm. Algebra}, 50(2):615--634, 2022.

\bibitem[Nat21]{MR4281372}
S. Natale.
\newblock On the notion of exact sequence: from {H}opf algebras to tensor
  categories.
\newblock In {\em Hopf algebras, tensor categories and related topics}, volume
  771 of {\em Contemp. Math.}, pages 225--254. Amer. Math. Soc., 2021.

\bibitem[Neg21]{MR4227163}
C. Negron.
\newblock Log-modular quantum groups at even roots of unity and the quantum
  {F}robenius {I}.
\newblock {\em Comm. Math. Phys.}, 382(2):773--814, 2021.

\bibitem[SS23]{2021arXiv211008739S}
T. {Shibata} and K. {Shimizu}.
\newblock {Nakayama functors for coalgebras and their applications for
  Frobenius tensor categories}.
\newblock {\em Adv. Math.}, 419, 2023.
\newblock {\tt arXiv:2110.08739}.

\bibitem[Ste75]{MR0389953}
B.~Stenstr\"{o}m.
\newblock {\em Rings of quotients}.
\newblock Die Grundlehren der mathematischen Wissenschaften, Band 217.
  Springer-Verlag, New York-Heidelberg, 1975.
\newblock An introduction to methods of ring theory.

\bibitem[Sul72]{MR304418}
J.~B. Sullivan.
\newblock Affine group schemes with integrals.
\newblock {\em J. Algebra}, 22:546--558, 1972.

\bibitem[Tak77]{MR472967}
M. Takeuchi.
\newblock Morita theorems for categories of comodules.
\newblock {\em J. Fac. Sci. Univ. Tokyo Sect. IA Math.}, 24(3):629--644, 1977.

\bibitem[Ven16]{MR3556434}
S. Venkatesh.
\newblock Hilbert basis theorem and finite generation of invariants in
  symmetric tensor categories in positive characteristic.
\newblock {\em Int. Math. Res. Not. IMRN}, (16):5106--5133, 2016.

\bibitem[{Ven}19]{2019arXiv190911240V}
S. {Venkatesh}.
\newblock {Harish-Chandra pairs in the Verlinde category in positive
  characteristic}.
\newblock {\tt arXiv:1909.11240}.

\bibitem[{Ven}22]{2022arXiv220303158V}
S. {Venkatesh}.
\newblock {Representations of General Linear Groups in the Verlinde Category}.
\newblock {\tt arXiv:2203.03158}.
\end{thebibliography}
\def\cprime{$'$}

\end{document}